\documentclass{amsart}%
\usepackage{amsmath}
\usepackage{amsfonts}
\usepackage{amssymb}
\usepackage{graphicx}%
\providecommand{\U}[1]{\protect\rule{.1in}{.1in}}
\newtheorem{theorem}{Theorem}

\newtheorem{lemma}[theorem]{Lemma}

\newtheorem{proposition}[theorem]{Proposition}

\begin{document}
\author[Yong Liu]{Yong Liu}
\address{Department of Mathematics, University of Science and Technology of China,
Hefei, China.}
\email{yliumath@ustc.edu.cn}
\author[Juncheng Wei]{Juncheng Wei}
\address{Department of Mathematics, University of British Columbia, Vancouver, B.C.
Canada, V6T 1Z2.}
\email{jcwei@math.ubc.ca}
\author[Wen Yang]{Wen Yang}
\address{Wuhan Institute of Physics and Mathematics, Innovation Academy for Precision
Measurement Science and Technology, Chinese Academy of Sciences, Wuhan, China.}
\email{math.yangwen@gmail.com}
\title{Uniqueness of lump solution to the KP-I equation}
\maketitle

\begin{abstract}
The KP-I equation has family of solutions which decay to zero at space
infinity. One of these solutions is the classical lump solution, which is a
traveling wave, and the KP-I equation in this case reduces to the Boussinesq
equation. In this paper we classify the lump type solutions of the Boussinesq
equation. Using a robust inverse scattering transform developed by
Bilman-Miller, we show that the lump type solutions are rational and their tau
functions have to be polynomials of degree $k(k+1)$ for some integer $k.$ In
particular, this implies that the lump solution is the unique ground state of
the KP-I equation(as conjectured by Klein-Saut). Our result can be regarded as
a two dimensional analogy of a theorem of Airault-McKean-Moser on the
classification of rational solutions for the KdV equation.

\end{abstract}

\section{Introduction}

The KP equation first appeared in the 1970 paper \cite{KP} by Kadomtsev and
Petviashvili, where they studied the transverse stability of the line solitons
of KdV equation. It can be written as
\[
\partial_{x}\left(  \partial_{t}U+\partial_{x}^{3}U+3\partial_{x}\left(
U^{2}\right)  \right)  -\sigma\partial_{y}^{2}U=0.
\]
Here $\sigma$ is a parameter and if $\sigma=1,$ then it is called KP-I
equation and has positive dispersion, while the case of $\sigma=-1$ is of
negative dispersion and called KP-II.

\bigskip

KP equation is an integrable system and can be regarded as a two dimensional
generalization of the classical KdV equation. It is an important PDE both in
mathematics and physics. Up to now, there exists vast literature on KP
equation. In the sequel, we shall briefly mention some results which are most
closely related to our objective.

\medskip

There are various different ways to study the KP equation. One of them is to
use the inverse scattering transform(IST for short). Manakov \cite{Manakov}
studied the IST of KP equation on a formal level. Segur in \cite{Segur} then
analyzed the direct scattering and rigorously obtained the solution for the
direct problem under a small norm assumption, but lump solutions are not
investigated in these works. Then Fokas-Ablowitz \cite{Fokas} obtained the
lump solution in their IST framework. Their results are later extended to
include higher order rational solutions in \cite{Vi}. Zhou \cite{Zhou, Zhou1,
Zhou2} then studied the KP-I equation and related problems in a more abstract
and rigorous way. There the lump solutions correspond to poles of the
associated eigenfunctions. Since the pole structure is still not well
understood in the general case, the lump solutions are actually not treated.
Later Boiti etc. have also studied the IST of KP-I in \cite{BPP}, with initial
data belonging to the Schwartz space. However, in spite of all these important
progresses, in general, the IST of KP-I equation is not completely understood yet.

\medskip

Observe that if $U$ is a traveling wave of the form $u\left(  x-t,y\right)  ,$
then the KP-I equation reduces to the following Boussinesq equation:
\begin{equation}
\partial_{x}^{2}\left(  \partial_{x}^{2}u+3u^{2}-u\right)  -\partial_{y}%
^{2}u=0.\label{Bous}%
\end{equation}
Due to the above mentioned difficulty, we would like to study the traveling
wave solutions of the KP-I equation using the IST of the Boussinesq equation,
which should, in principle, be easier than the KP case. The IST of Boussinesq
equation is first carried out in \cite{Deift}. The first equation of the
associated Lax pair turns out to be a third order ODE, in contrast with the
second order ODE for the KdV case. In this direction, there are some related
works. For instances, the IST for first order ODE systems with generic
potentials(means that the poles are all simple) has been studied in
\cite{Beals1}, \cite{Beals2} and the case of higher order ODEs has been
treated in \cite{Beals}. The case of general potentials has been studied in
\cite{DeiftZh, Zhou1} using the augmented contour approach. Recently, the
hyperbolic case of the Boussinesq equation(so called \textquotedblleft
good\textquotedblright\ Boussinesq equation ) is studied in \cite{Char1},
\cite{Char2} using Riemann-Hilbert approach. For Schwartz class initial data,
long time dynamics is obtained. Note that in this case, the equation does not
have lump solution. 

\medskip

Let us write the solution $u$ of (\ref{Bous}) in terms of the $\tau$ function:
$u:=2\partial_{x}^{2}\ln\tau.$ Then the Boussinesq equation in bilinear form
is
\begin{equation}
\left(  \mathfrak{D}_{x}^{4}-\mathfrak{D}_{x}^{2}-\mathfrak{D}_{y}^{2}\right)
\tau\cdot\tau=0.\label{k}%
\end{equation}
Throughout the paper, we will use the symbol $\mathfrak{D}$ to denote the
bilinear derivative operator. We refer to the classical book by Hirota
\cite{Hirota} for detailed exposition to the bilinear derivative operator and
the direct method in soliton theory, including that of the KP equation. One
can check that the function $\tau\left(  x,y\right)  =x^{2}+y^{2}+3$ is a
solution of the bilinear equation $\left(  \ref{k}\right)  $. This function is
even in both $x$ and $y$ variables and corresponding to the classical lump
solution. The lump solution is a rogue wave extensively studied before and is
a special one in the large class of lump type solutions, whose precise
definition will be given below. Note that actually the lump solution is first
obtained in \cite{Manakovz,Sat} using a limiting procedure. The spectral
property of lump solution is now well understood. Indeed, the first and second
authors have proved in \cite{Liu} using the Backlund transformation that the
it is nondegenerated, in the sense that the linearized KP-I operator at this
solution does not have any nontrivial kernels. This also implies that the lump
is orbitally stable.

\bigskip

The importance of KP equation is also reflected by the fact that it appears in
the study of many other PDEs. For instance, in \cite{Be}(see also the
references therein), it is shown that KP equation is related to the GP
equation. The nondegeneracy result for the lump can then be used to construct
traveling wave solutions of the GP equation with subsonic speed, with a
perturbation argument, see \cite{LWW}.

\medskip

More general rational solutions of $\left(  \ref{k}\right)  $ with degree
$k\left(  k+1\right)  $ have been found in \cite{Pelinovskii95},
\cite{Pelinovskii}. Then in \cite{Pelinovskii93} it is proved that around the
higher energy lump type solutions, the KP-I equation has anomalous scattering
with infinite phase shift. This indicates that the dynamics of the KP-I
equation will be more complicated than the KdV equation. Hence it is important
to understand the structure of lump type solutions for the Boussinesq, as well
as the KP-I equation. We also point out that KP-I is globally well-posed in
the natural energy space, see, for instance \cite{Kenig},\cite{Moli}%
,\cite{Moli2} and the references therein for related results in this
direction. A very fascinating and in-depth description of KP equation and
related dynamical, variational, and other properties of its solutions,
including lump, can be found in the book of Klein-Saut \cite{Klein}. We refer
to it and also its references for a detailed introduction on this subject.

\medskip

It is worth mentioning that $\left(  \ref{Bous}\right)  $ is a special case of
Boussinesq-type equation(its original form described by Boussinesq in 1870s)
\[
\partial_{x}^{2}\left(  pu+u^{2}+\partial_{x}^{2}u\right)  +\sigma^{2}%
\partial_{y}^{2}u=0,
\]
where $\sigma^{2}=\pm1$ and $p$ is a constant. Rational solutions of this
equation has been studied in many papers, such as \cite{Airault, Moser, Bog,
Clarkson, Pelinovskii95, H}. For instance, the special case of $\left(
\sigma,p\right)  =\left(  1,0\right)  $ is considered in \cite{Clarkson},
using the theory of Painlev\'{e} equations. Most of these works are concerned
with the construction of explicit solutions and the analysis of their
mathematical or physical properties.

\medskip

In view of all these developments, it is desirable to have some classification
on the solutions of the Boussinesq equation. In this paper, we would like to
classify all the \textquotedblleft lump type\textquotedblright\ solutions. Our
first result is the following

\begin{theorem}
\label{main}Suppose $u$ is a real valued $C^{4}$ solution of the equation
\[
\partial_{x}^{2}\left(  \partial_{x}^{2}u+3u^{2}-u\right)  -\partial_{y}%
^{2}u=0\text{ in }\mathbb{R}^{2}.
\]
Assume that there exists $\alpha>0$ such that
\[
\left\vert u\left(  x,y\right)  \right\vert \leq\frac{C}{\left(  1+x^{2}%
+y^{2}\right)  ^{\alpha}}.
\]
Then $u=2\partial_{x}^{2}\ln\tau_{k},$ where $\tau_{k}$ is a polynomial in
$x,y$ of degree $k\left(  k+1\right)  $ for some $k\in\mathbb{N}$.
\end{theorem}

We remark that the assumption that $u$ is real valued is essential, since
there are complex valued solutions whose tau functions have degree different
from $k(k+1)$.

\medskip

If we consider those solutions which are even, then we have

\begin{theorem}
Suppose $\tau$ is a polynomial of degree $2n$ with real coefficients
satisfying
\[
\tau\left(  x,y\right)  =\tau\left(  x,-y\right)  =\tau\left(  -x,y\right)
\]
and
\[
\left(  \mathfrak{D}_{x}^{2}+\mathfrak{D}_{y}^{2}-\mathfrak{D}_{x}^{4}\right)
\tau\cdot\tau=0.
\]
Assume $n=k\left(  k+1\right)  /2\leq300$ for some positive integer $k.$ Then
$\tau$ is unique, up to a multiplicative constant.
\end{theorem}

We would like to emphasize that the upbound $300$ is not optimal, and we
expect the uniqueness of even solution holds for all $n=k\left(  k+1\right)
/2.$ We refer to Section $5$ for more details.

\medskip

The KP-I equation has a variational structure. As a corollary of Theorem
\ref{main}, we see that the classical lump solution is the unique ground state
of the KP-I equation, due to the fact that the energy is determined by the
degree of the tau function. This answers affirmatively the uniqueness question
raised in Remark 18 and Remark 19 by Klein-Saut in \cite{Klein0}. As already
pointed out there, while the uniqueness of ground state of the Schrodinger
equation can be proved using ODE shooting method, the uniqueness of lump is
more complicated since it is not radially symmetric. We have in mind that
those even travelling wave solutions of the KP-I equation should play similar
role as the radially symmetric solutions of the Schrodinger equation. To our
knowledge, our result seems to be the first classification result for
solutions of semilinear elliptic equations without symmetry(also without any
other assumptions like stable or finite Morse index).

\medskip

Solutions satisfying the assumptions of Theorem \ref{main} will be called lump
type solutions. We remark that for each fixed $k,$ there is a family of lump
type solutions, found in \cite{Pelinovskii95}. We expect that all lump type
solutions should be included in this family.  Those solutions will be recalled
in the next section. However, a full classification of this type would need
further detailed analysis, which will not be pursued in this paper. Such a
full classification would presumably yields some information of the lump type
solutions of the generalized KP equation.

\bigskip

Many questions remain to be answered. For instances, the computation of the
Morse indices of lump type solutions, the asymptotical stability of the lump
solution, the classification of solutions to the general Boussinesq equation
with zero or nonzero condition near infinity, etc. Note that by a result of
\cite{Saut97}, Pohozaev type identity tells us that the KP-II equation does
not have lump type solutions.

\medskip

Let us now sketch the main ideas of our proof. We first use the robust inverse
scattering transform developed by Bilman-Miller \cite{Bilman} to show that
lump type solutions has to be rational. Then we use the technique of
\cite{Moser}, appealing to the Boussinesq hierarchy, to show that the degree
of the $\tau$ function has to be $k\left(  k+1\right)  .$ This technique is
used in \cite{Moser} to prove that the $\tau$ function of the rational
solution of the KdV equation necessarily is a polynomial of degree $k\left(
k+1\right)  /2.$ We hope that our method should also be applicable to other
integrable systems such as 2d Toda lattice.\medskip

\medskip

This paper is organized in the following way. In Section 2, we recall the
construction of lump type solutions appeared in various papers of Pelinovskii
and his collaborators. We emphasize that the KP-I equation is a well studied
model equation and actually there are many other constructions, using
different methods. In Section 3, we use the robust inverse scattering
transform to show that lump type solutions are rational. We then investigate
the degree of the $\tau$ functions in Section 4. The last section is devoted
to analyze the even solutions. In particular, we prove that even solution is
unique when the degree of its tau function is not so large. 

\medskip

\textbf{Acknowledgment.} Y. Liu is supported by the National Key R\&D Program
of China 2022YFA1005400 and NSFC 11971026, NSFC 12141105. J. Wei is partially
supported by NSERC of Canada. W. Yang is partially supported by National Key
R\&D Program of China 2022YFA1006800, NSFC No. 12171456 and NSFC No. 12271369.

\section{Family of lump type solutions}

In this section, we recall the construction of lump type solutions. Although
the materials in this section will not be used in the proof of our main
results, it will be helpful to provide a rough picture of what lump solutions
should be. 

Real valued rational solutions of the Boussinesq equation whose tau functions
have degree $k\left(  k+1\right)  $ have been obtained in \cite{Pelinovskii}
by a limiting procedure. These solutions are even with respect to $x$ and $y$
variable. For instance, the function $12\partial_{x}^{2}\ln\tau,$ where
\[
\tau=\left(  x^{2}+y^{2}\right)  ^{3}+25x^{4}+90x^{2}y^{2}+17y^{4}%
-125x^{2}+475y^{2}+1875,
\]
solves the equation%
\begin{equation}
\left(  -u+\frac{u^{2}}{2}+u_{xx}\right)  _{xx}-u_{yy}=0.\label{BouofP}%
\end{equation}
Observe that the coefficients in equation $\left(  \ref{BouofP}\right)  $ are
different from that of $\left(  \ref{Bous}\right)  .$ However, they can be
transformed to one other simply by suitable rescaling of the form $au\left(
bx,cy\right)  .$ In the rest of papers, we will also consider the Boussinesq
equation with other coefficients in different contexts. This is to make them
to be consistent with the corresponding literature.

In \cite{Pelinovskii95},\ more general families of rational solutions have
been derived using Wronskian representation of the solutions to the KP
equation. Among these solutions, those travelling waves reduces to the
Boussinesq equation. Let us recall these results in the sequel. We adopt the
notations used in \cite{Pelinovskii95}.

Consider the KP-I equation in the following form:
\[
\left(  -4u_{t_{3}}+\left(  3u^{2}\right)  _{t_{1}}+u_{t_{1}t_{1}t_{1}%
}\right)  _{t_{1}}-3u_{t_{2}}u_{t_{2}}=0.
\]
This equation has solutions expressed in terms of the $\tau$ function:
\[
u\left(  t_{1},t_{2},t_{3}\right)  =2\partial_{t_{1}}^{2}\ln\tau\left(
t_{1},t_{2},t_{3}\right)  .
\]
There are different forms for the $\tau$ functions. Let us explain it now.

Let $\Psi_{n}^{\pm}$ be solutions of the system of differential equations
\[
\left\{
\begin{array}
[c]{c}%
\pm i\partial_{t_{2}}\Psi_{n}^{\pm}=\partial_{t_{1}}^{2}\Psi_{n}^{\pm},\\
\partial_{t_{3}}\Psi_{n}^{\pm}=\partial_{t_{1}}^{3}\Psi_{n}^{\pm}.
\end{array}
\right.
\]
We fix an integer $N$ and define
\begin{equation}
\tau=\det M_{N},\label{tau1}%
\end{equation}
where $M_{N}$ is the $N\times N$ matrix whose entries are given by
$c_{nk}+I_{nk},1\leq n,k\leq N.$ Here $c_{nk}$ are arbitrary complex
parameters and
\[
I_{nk}=\int_{-\infty}^{t_{1}}\Psi_{n}^{+}\left(  s,t_{2},t_{3}\right)
\Psi_{k}^{-}\left(  s,t_{2},t_{3}\right)  ds.
\]
With this definition, the function $2\partial_{t_{1}}^{2}\ln\tau$ is a
solution of the KP-I equation. 

The KP-I equation has another family of solutions, for which the $\tau$
function has the Wronskian form:
\begin{equation}
\tau=W\left(  \Psi_{1}^{\pm},...,\Psi_{N}^{\pm}\right)  =\det\left(
J_{nk}^{\pm}\right)  ,\label{tau2}%
\end{equation}
where $J_{nk}^{\pm}=\partial_{t_{1}}^{k-1}\Psi_{n}^{\pm}.$

The two forms $\left(  \ref{tau1}\right)  $ and $\left(  \ref{tau2}\right)  $
are related to each other. If we choose in $\left(  \ref{tau1}\right)  $ the
function
\[
\Psi_{k}^{-}=\exp\left(  p_{k}t_{1}-p_{k}^{2}t_{2}+p_{k}^{3}t_{3}\right)
:=\exp\left(  \Phi_{k}^{-}\right)  ,
\]
then integration by parts yields
\[
I_{nk}=\left(  \frac{\Psi_{n}^{+}}{p_{k}}-\frac{\partial_{t_{1}}\Psi_{n}^{+}%
}{p_{k}^{2}}+\frac{\partial_{t_{1}}^{2}\Psi_{n}^{+}}{p_{k}^{3}}+...\right)
\exp\left(  \Phi_{k}^{-}\right)  .
\]

Assuming $p_{k}>>1,$ the leading terms of $\tau$ can be written as the product
of a Vandermont determinant and the Wronskian $W\left(  \Psi_{1}^{+}%
,...,\Psi_{N}^{+}\right)  .$ Hence
\[
\tau=\left[  \frac{%
%TCIMACRO{\dprod \limits_{1\leq m<k\leq N}}%
%BeginExpansion
{\displaystyle\prod\limits_{1\leq m<k\leq N}}
%EndExpansion
\left(  p_{k}-p_{m}\right)  }{%
%TCIMACRO{\dprod \limits_{k=1}^{N}}%
%BeginExpansion
{\displaystyle\prod\limits_{k=1}^{N}}
%EndExpansion
p_{k}^{N}}W\left(  \Psi_{1}^{+},...,\Psi_{N}^{+}\right)  +O\left(  p^{-\left(
\frac{N\left(  N+1\right)  }{2}+1\right)  }\right)  \right]  \exp\left(
%TCIMACRO{\dsum \limits_{k=1}^{N}}%
%BeginExpansion
{\displaystyle\sum\limits_{k=1}^{N}}
%EndExpansion
\Phi_{k}^{-}\right)  .
\]
Dividing the right hand side by $\exp\left(
%TCIMACRO{\dsum \limits_{k=1}^{N}}%
%BeginExpansion
{\displaystyle\sum\limits_{k=1}^{N}}
%EndExpansion
\Phi_{k}^{-}\right)  $ and the constant before the Wronskian $W,$ and letting
$p\rightarrow0,$ we get $\left(  \ref{tau2}\right)  .$

Now let $K\leq N$ be a fixed integer. If the above limiting procedure is only
carried out for $\Phi_{k}^{-},k=K+1,...,N,$ then we obtain
\begin{equation}
\tau=\det\left(  S_{nk}\right)  ,\label{tau3}%
\end{equation}
where
\[
S_{nk}=\left\{
\begin{array}
[c]{l}%
I_{nk},\text{ for }k=1,...,K,\\
J_{n,k-K}^{+},\text{ for }k=K+1,...,N.
\end{array}
\right.
\]
Let us now consider the function $\phi_{m}:=\partial_{p}^{m}\exp\left(
\Phi^{+}\left(  t_{1},t_{2},t_{3},p\right)  \right)  ,$ where
\[
\Phi^{+}\left(  t_{1},t_{2},t_{3},p\right)  =%
%TCIMACRO{\dsum \limits_{j=1}^{\infty}}%
%BeginExpansion
{\displaystyle\sum\limits_{j=1}^{\infty}}
%EndExpansion
\left(  p^{j}t_{j}\right)  .
\]
We have
\[
\phi_{m}=P_{m}\exp\left(  \Phi^{+}\left(  t_{1},t_{2},t_{3},p\right)  \right)
.
\]
Here $P_{m}$ is a polynomial of the variables $\theta_{1},...,\theta_{m},$
given by $\theta_{j}=\frac{1}{m!}\partial_{p}^{j}\Phi^{+}.$ In particular,
\begin{align*}
\theta_{1} &  =t_{1}+2pt_{2}+3p^{2}t_{3}+...,\\
\theta_{2} &  =t_{2}+3pt_{3}+...\\
\theta_{3} &  =t_{3}+...,
\end{align*}
and $\theta_{j}$ only depends on $t_{j},t_{j+1},....$ We have
\[
P_{1}=\theta_{1},P_{2}=2\theta_{2}+\theta_{1}^{2}.
\]

To obtain a solution of the Boussinesq equation, let $v=-\frac{1}{3p}$ and
define the vertex operator
\[
\mathcal{S}\left(  v\right)  =\exp\left(  -
%TCIMACRO{\dsum \limits_{m=1}^{+\infty} }%
%BeginExpansion
{\displaystyle\sum\limits_{m=1}^{+\infty}}
%EndExpansion
\frac{v^{m}}{m}\partial_{\theta_{m}}\right)  .
\]
Then
\[
\mathcal{S}\left(  v\right)  P_{n}=\left(  1-v\partial_{t_{1}}\right)  P_{n}.
\]
The tau function will be a solution of the Boussinesq equation if it depends
on the variable $x=t_{1}+3p^{2}t_{3}$ and $t_{2}.$ This requires
\[
\partial_{\theta_{2}}\tau=v\partial_{\theta_{3}}\tau.
\]

The next step to construct solutions of the Boussinesq equation is to define
\begin{align*}
\Psi_{n}^{+} &  =\left(  S^{N-n}\left(  v\right)  P_{2n-1}^{+}\right)
\exp\left(  \Phi^{+}\left(  t_{1},t_{2},t_{3},p\right)  \right)  ,\text{ for
}1\leq n\leq N,\\
\Psi_{k}^{-} &  =\left(  S^{K-k}\left(  v\right)  P_{2k-1}^{-}\right)
\exp\left(  \Phi^{-}\left(  t_{1},t_{2},t_{3},p\right)  \right)  ,1\leq k\leq
K.
\end{align*}
Then the tau function defined by $\left(  \ref{tau3}\right)  $ will correspond
to a solution of the Boussinesq equation. In general, this solution is complex
valued. But in the particular case of $K=N,$ if we choose $P_{k}^{+},P_{k}%
^{-}$ such that $P_{k}^{+}=\bar{P}_{k}^{-},$ then
\[
\tau=\det\left(  w^{+}\left(  w^{-}\right)  ^{T}\right)  ,
\]
where $\mu=-\frac{1}{2p}$ and
\begin{align*}
\left(  w^{+}\right)  _{nk} &  =\left(  -\mu\right)  ^{k-1}\partial_{t_{1}%
}^{k-1}\left[  S^{-k}\left(  \mu\right)  S^{N-n}\left(  v\right)  P_{2n-1}%
^{+}\right]  ,1\leq n\leq N,1\leq k\leq2N-1,\\
w_{nk}^{-} &  =\bar{w}_{nk}^{+}.
\end{align*}
Hence this determinant can be written as the sum of positive terms. Note that
the condition $P_{k}^{+}=\bar{P}_{k}^{-}$ requires $t_{2k}$ to be imaginary
and $t_{2k+1}$ to be real. Hence there are in total $2N$ free
(real)parameters, or $N$ complex parameters.

We point out that in \cite{Yang}, an explicit family of rational solutions is
also obtained with different methods. In the degree $6$ case, the family of
functions $2\partial_{x}^{2}\ln\tau,$ where
\begin{align*}
\tau\left(  x,y\right)   &  =x^{6}+y^{6}+3x^{4}y^{2}+3x^{2}y^{4}%
+14x^{5}+14xy^{4}+28x^{3}y^{2}+90x^{4}\\
&  +128x^{2}y^{2}+22y^{4}+324x^{3}+316xy^{2}+648x^{2}+360y^{2}+648x+324\\
&  +2a\left(  x^{3}-3xy^{2}+7x^{2}-7y^{2}+16x+8\right)  \\
&  +2by\left(  y^{2}-3x^{2}-14x-18\right)  +a^{2}+b^{2},
\end{align*}
with $a,b$ being real valued parameters, solve the following Boussinesq
equation
\[
\left(  -3u+3u^{2}+u_{xx}\right)  _{xx}+3u_{yy}=0.
\]
It is also worth mentioning that there already exist many papers on the
construction and analysis of solutions to the KP and related equations, for
instances, \cite{Ablowitz, Cha0,Cha1,Prada, W, Ito, Kac, Pelnovsky94,
Pelnovsky98}, just to list a few of them.

\section{Inverse scattering of the Boussinesq equation and the rationality of
lump type solutions}

We would like to show that lump type solutions of the Boussinesq equation have
to be rational functions. The equation to be studied in this section reads as
\begin{equation}
q_{yy}=3q_{xxxx}-12\left(  q^{2}\right)  _{xx}-24q_{xx}. \label{B}%
\end{equation}
It can be obtained from the original Boussinesq equation $\left(
\ref{Bous}\right)  $ by a simple rescaling, that is, by setting $q\left(
x,y\right)  =-6u\left(  2\sqrt{2}x,8\sqrt{3}y\right)  .$

Observe that every constant function solves (\ref{B}). Here we will focus on
the special class of solutions decaying to \textit{zero} at infinity. The
usual inverse scattering of the Boussinesq equation is developed in
\cite{Deift}, with a nonzero boundary condition near infinity. It can be seen
later on that in our case, the situation is much more complicated, since the
corresponding fundamental solutions have singularities in the complex plane of
spectral parameter. To overcome these difficulties, we will adopt the powerful
method of \textquotedblleft robust\textquotedblright\ inverse scattering
transform to show that lump type solutions of (\ref{B}) have to be rational.
This type of robust inverse scattering has first been developed in
\cite{Bilman} for the Schrodinger equation.

\subsection{Refined asymptotics of lump type solutions}

To carry out the robust inverse scattering transform, it turns out to be
important to get a precise decay estimate for the lump type solutions.

We would like to prove the following refined asymptotics result.

\begin{proposition}
\label{refi}Suppose $u$ is a real valued $C^{4}$ solution of the Boussinesq
equation
\[
\partial_{x}^{2}\left(  \partial_{x}^{2}u+3u^{2}-u\right)  -\partial_{y}%
^{2}u=0\text{ in }\mathbb{R}^{2}.
\]
Assume that for some $\alpha>0,$
\begin{equation}
\left\vert u\left(  x,y\right)  \right\vert \leq\frac{C}{\left(  1+x^{2}%
+y^{2}\right)  ^{\alpha}}. \label{udecay}%
\end{equation}
Then there holds
\begin{equation}
\left\vert u\left(  x,y\right)  \right\vert \leq\frac{C}{1+x^{2}+y^{2}}.\text{
} \label{decay}%
\end{equation}

\end{proposition}

\begin{proof}
The solution $u$ satisfies
\[
\partial_{x}^{4}u-\partial_{x}^{2}u-\partial_{y}^{2}u=-3\partial_{x}%
^{2}\left(  u^{2}\right)  .
\]
Then
\[
u\left(  x,y\right)  =3\int_{\mathbb{R}^{2}}K\left(  x-s,y-t\right)
u^{2}\left(  s,t\right)  dsdt,
\]
where the kernel $K$ is defined through the Fourier transform:
\[
K\left(  x,y\right)  =\int_{\mathbb{R}^{2}}\frac{\xi_{1}^{2}}{\xi_{1}^{4}%
+\xi_{1}^{2}+\xi_{2}^{2}}e^{ix\xi_{1}+iy\xi_{2}}d\xi_{1}d\xi_{2}.
\]
By Lemma 3.6 of \cite{Bouard97}, we have
\begin{equation}
\left(  x^{2}+y^{2}\right)  K\in L^{\infty}\left(  \mathbb{R}^{2}\right)  .
\label{estiK}%
\end{equation}
To simplify the notation, we introduce $r=\sqrt{x^{2}+y^{2}}$ and define
\[
\Omega_{1}:=\left\{  \left(  s,t\right)  :\left(  s-x\right)  ^{2}+\left(
t-y\right)  ^{2}\leq\frac{r^{2}}{4}\right\}  ,
\]%
\[
\Omega_{2}:=\left\{  \left(  s,t\right)  :s^{2}+t^{2}\leq\frac{r^{2}}%
{4}\right\}  .
\]
Using $\left(  \ref{udecay}\right)  ,\left(  \ref{estiK}\right)  $ and the
integrability of the kernel $K\left(  s,t\right)  $ around $\left(
0,0\right)  ,$ we can estimate
\begin{align*}
&  \int_{\Omega_{1}}\left\vert K\left(  x-s,y-t\right)  \right\vert
u^{2}\left(  s,t\right)  dsdt\\
&  \leq\int_{\Omega_{1}}\frac{\left\vert K\left(  x-s,y-t\right)  \right\vert
}{\left(  1+x^{2}+y^{2}\right)  ^{2\alpha}}dsdt\\
&  \leq\frac{C\ln\left(  2+r\right)  }{\left(  1+x^{2}+y^{2}\right)
^{2\alpha}}.
\end{align*}
We also have
\begin{align*}
&  \int_{\Omega_{2}}\left\vert K\left(  x-s,y-t\right)  \right\vert
u^{2}\left(  s,t\right)  dsdt\\
&  \leq\int_{\Omega_{2}}\frac{u^{2}\left(  s,t\right)  }{1+x^{2}+y^{2}}dsdt\\
&  \leq\frac{C}{1+x^{2}+y^{2}}+\frac{C}{\left(  1+x^{2}+y^{2}\right)
^{2\alpha}}.
\end{align*}
Moreover,
\begin{align*}
&  \int_{\mathbb{R}^{2}\backslash\left(  \Omega_{1}\cup\Omega_{2}\right)
}\left\vert K\left(  x-s,y-t\right)  \right\vert u^{2}\left(  s,t\right)
dsdt\\
&  \leq\int_{\mathbb{R}^{2}\backslash\left(  \Omega_{1}\cup\Omega_{2}\right)
}\frac{Cdsdt}{\left(  1+s^{2}+t^{2}\right)  ^{2+2\alpha}}\\
&  \leq\frac{C}{\left(  1+x^{2}+y^{2}\right)  ^{2\alpha}}.
\end{align*}
Combining all these estimates, we deduce
\[
\left\vert u\left(  x,y\right)  \right\vert \leq\frac{C}{1+x^{2}+y^{2}}%
+\frac{C}{\left(  1+x^{2}+y^{2}\right)  ^{\frac{3\alpha}{2}}}.
\]
A straightforward bootstrapping argument tells us that
\[
\left\vert u\left(  x,y\right)  \right\vert \leq\frac{C}{1+x^{2}+y^{2}}.
\]
This is the required decay estimate.
\end{proof}

The estimate $\left(  \ref{decay}\right)  $ is optimal, as can be seen from
the classical lump solution and the examples discussed in the previous
section. Note that in \cite{Bouard97}, optimal decay estimates same as
$\left(  \ref{decay}\right)  $ are obtained under the assumption that the
solution is integrable in suitable sense, that is, belongs to the natural
energy space. We point out that the estimate of Proposition \ref{refi}
actually still holds if we only assume that $u$ tends to zero at infinity,
without any a priori algebraic decay rate assumption. However, the proof of
this will be quite delicate, since at the beginning we don't have any decay
estimate for $u.$ We leave it for further study.

\subsection{Inverse Scattering}

Introducing a new function $p$, the equation $\left(  \ref{B}\right)  $ can
also be written into the following system of ODEs:
\[
\left\{
\begin{array}
[c]{l}%
q_{y}=-3p_{x},\\
p_{y}=-q_{xxx}+8qq_{x}+8q_{x}.
\end{array}
\right.
\]
This system is corresponding to the following Lax pair equation(see
\cite{Deift, Zah}):
\[
\frac{dL}{dy}=QL-LQ=[Q,L],
\]
where
\[
\left\{
\begin{array}
[c]{l}%
L=i\frac{d^{3}}{dx^{3}}-i\left[  \left(  2\left(  q+1\right)  \frac{d}%
{dx}+q_{x}\right)  \right]  +p,\\
Q=i\left(  3\frac{d^{2}}{dx^{2}}-4\left(  q+1\right)  \right)  .
\end{array}
\right.
\]
Here $i$ is the imaginary unit. Let $k\in\mathbb{C}\,\ $be a complex spectral
parameter. We consider the equation
\begin{equation}
Lf=\left(  k^{3}+2k\right)  f.\label{Lax1}%
\end{equation}
Introducing vector $\mathbf{f:=}\left(  f_{1},f_{2},f_{3}\right)  ^{T}$ by
$f_{1}=f,f_{2}=f_{1}^{\prime},f_{3}=f_{2}^{\prime},$ we obtain the following
system of ODEs:
\begin{equation}
\frac{d}{dx}\left(
\begin{array}
[c]{c}%
f_{1}\\
f_{2}\\
f_{3}%
\end{array}
\right)  =\left(
\begin{array}
[c]{ccc}%
0 & 1 & 0\\
0 & 0 & 1\\
q_{x}+pi-i\left(  k^{3}+2k\right)   & 2\left(  q+1\right)   & 0
\end{array}
\right)  \left(
\begin{array}
[c]{c}%
f_{1}\\
f_{2}\\
f_{3}%
\end{array}
\right)  .\label{firstorder}%
\end{equation}
The coefficient matrix of this system, denoted by $A,$ will depend on the
potential $q$ and $p.$ As $x\rightarrow\pm\infty,$ $A$ will tend to the
following trace-free constant matrix
\[
T:=\left(
\begin{array}
[c]{ccc}%
0 & 1 & 0\\
0 & 0 & 1\\
-i\left(  k^{3}+2k\right)   & 2 & 0
\end{array}
\right)  .
\]
$\allowbreak$ The eigenvalues of $T$ can be explicitly computed. They depend
on the parameter $k$ and are given by
\[
\lambda_{1}=ik,\quad\lambda_{2}=\frac{-ik+\sqrt{3k^{2}+8}}{2},\quad\lambda
_{3}=\frac{-ik-\sqrt{3k^{2}+8}}{2}.
\]
It follows that $T$ can be written as $PMP^{-1}$, where
\[
P\left(  k\right)  =\left(
\begin{matrix}
1 & 1 & 1\\
ik & \frac{-ik+\sqrt{3k^{2}+8}}{2} & \frac{-ik-\sqrt{3k^{2}+8}}{2}\\
-k^{2} & \frac{k^{2}+4-ik\sqrt{3k^{2}+8}}{2} & \frac{k^{2}+4+ik\sqrt{3k^{2}%
+8}}{2}%
\end{matrix}
\right)  ,M(k)=\left(
\begin{matrix}
\lambda_{1} & 0 & 0\\
0 & \lambda_{2} & 0\\
0 & 0 & \lambda_{3}%
\end{matrix}
\right)  .
\]

Recall that for any constant matrix $B,$ the matrix $e^{Tx}B$ is a solution of
the equation $U^{\prime}=TU.$ We choose $B=P$ and get the following matrix
solution
\[
U_{bg}(k,x):=n(k)P\left(  k\right)  e^{M(k)x}:=E\left(  k\right)  e^{M\left(
k\right)  x}.
\]
Here $n\left(  k\right)  $ is chosen such that $\det\left(  U_{bg}\right)
=1$. Explicitly,
\[
n\left(  k\right)  =\left(  \left(  3k^{2}+2\right)  \sqrt{3k^{2}+8}\right)
^{-1}.
\]
One can see that as $k$ tends to $\pm\frac{\sqrt{6}}{3}i$ or $\pm\frac
{2\sqrt{6}}{3}i$, the function $n$ will blow up. For $j=1,2,3,$ let us denote
the $j$-th column of $E\left(  k\right)  $ by $\xi_{j}.$

Let $k=s+ti$, with $s,t\in\mathbb{R}.$ Direct computation tells us that the
condition $\operatorname{Re}\left(  \lambda_{2}\right)  =\operatorname{Re}%
\left(  \lambda_{3}\right)  $ implies $\operatorname{Re}\sqrt{3k^{2}+8}=0.$
That is,
\[
s=0\text{ and }t^{2}>\frac{8}{3}.
\]
On the other hand, $\operatorname{Re}\left(  \lambda_{1}\right)
=\operatorname{Re}\left(  \lambda_{2}\right)  $ or $\operatorname{Re}\left(
\lambda_{3}\right)  $ requires
\[
s^{2}-3t^{2}+2=0.
\]

Let $r$ be a fixed large constant and $B_{r}$ be the ball of radius $r$
centered at the origin. In the region $B_{r}^{c}:=\mathbb{R}^{2}\backslash
B_{r},$ we consider the curve
\[
\Sigma_{1}:=\left\{  \left(  s,t\right)  \in B_{r}^{c}:s^{2}-3t^{2}
+2=0\right\}  \cup\left\{  \left(  s,t\right)  \in B_{r}^{c}:s=0\text{ and
}t^{2}>\frac{8}{3}\right\}  .
\]
Let us define
\[
\Omega_{1}=B_{r}^{c}\backslash\Sigma_{1}.
\]
Note that $\Omega_{1}$ has six connected components, which we will denote them
by $\Omega_{1,1},...,\Omega_{1,6}.$

In the ball $B_{r},$ we consider the curve
\[
\Sigma_{2}:=\left\{  \left(  s,t\right)  :s=0,t^{2}\leq\frac{8}{3}\right\}  .
\]
We also define $\Omega_{2}:=B_{r}\backslash\Sigma_{2}.$

Next we define the a distinguished solution matrix for $\left(
\ref{firstorder}\right)  .$ Note that if the matrix $\phi$ satisfies
$\phi^{\prime}=A\phi,$ then $g:=\phi e^{-Mx}$ will satisfy
\begin{align*}
g^{\prime} &  =\phi^{\prime}e^{-Mx}+\phi\left(  -M\right)  e^{-Mx}\\
&  =Ag-gM.
\end{align*}
For $k\in\mathbb{R}^{2}\backslash B_{r},$ we choose to be the matrix solution
such that
\begin{equation}
\left\Vert g\left(  x\right)  \right\Vert _{L^{\infty}\left(  \mathbb{R}%
\right)  }<+\infty,\text{ and }g\left(  x\right)  \rightarrow E\left(
k\right)  ,\text{ as }x\rightarrow+\infty.\label{Beals}%
\end{equation}
We then define $U^{ou}=ge^{Mx}.$ The solution $\phi=ge^{Mx}$ satisfies
$\left(  \ref{Beals}\right)  $ is called Beals-Coifman fundamental solution.
The existence of this solution is explained in \cite{Beals}, Page 8. We will
sketch the main steps below. As is pointed out there, the first step is to
construct a solution with prescribed asymptotics at $-\infty,$ using the
arguments of \cite{Coddington}(Page 104, Problem 29). Since this construction
will play an important role later on, we recall the precise statement of the
result and its proof in the following

\begin{lemma}
\label{conBC}Assume $\lambda_{j},j=1,2,3,$ are distinct. Then the equation
\begin{equation}
\phi^{\prime}=A\phi\label{ode}%
\end{equation}
has a solution $\phi_{j}^{+}$ satisfying
\[
\phi_{j}^{+}\left(  x\right)  e^{-\lambda_{j}x}\rightarrow\xi_{j},\text{ as
}x\rightarrow+\infty.
\]
Similarly, $\left(  \ref{ode}\right)  $ also has a solution $\phi_{j}^{-}$
with
\[
\phi_{j}^{-}\left(  x\right)  e^{-\lambda_{j}x}\rightarrow\xi_{j},\text{ as
}x\rightarrow-\infty.
\]

\end{lemma}

\begin{proof}
Let $\operatorname{Re}\lambda_{j}=\sigma$ and $e^{Tx}=Y_{1}\left(  x\right)
+Y_{2}\left(  x\right)  ,$ where the entries of $Y_{1}$ are linear combination
of $e^{\lambda_{k}x}$ with $\operatorname{Re}\lambda_{k}<\sigma,$ and the
entries of $Y_{2}$ are linear combination of $e^{\lambda_{k}x}$ with
$\operatorname{Re}\lambda_{k}\geq\sigma.$ Thanks to the assumption that
$\lambda_{j}$ are distinct, this decomposition always exists.

We use a Picard iteration scheme and set $\eta_{0}\left(  x\right)
=e^{\lambda_{j}x}\xi_{j}$. Let $a$ be a fixed constant. Then we can define the
sequence $\left\{  \eta_{l}\right\}  $ by
\[
\eta_{l+1}\left(  x\right)  :=e^{\lambda_{j}x}\xi_{j}+\int_{a}^{x}Y_{1}\left(
x-s\right)  R\left(  s\right)  \eta_{l}\left(  s\right)  ds-\int_{x}^{+\infty
}Y_{2}\left(  x-s\right)  R\left(  s\right)  \eta_{l}\left(  s\right)  ds.
\]
The definition of $Y_{2}$ ensures that the last integral is well defined. Note
that when $x\leq0,$ there holds $\left\vert Y_{2}\left(  x\right)  \right\vert
\leq K_{2}e^{\sigma x}$ for some constant $K_{2}.$ Now if we assume
$\left\vert \eta\left(  s\right)  \right\vert \leq Ce^{\sigma s},$ then there
holds
\begin{align*}
\left\vert \int_{x}^{+\infty}Y_{2}\left(  x-s\right)  R\left(  s\right)
\eta\left(  s\right)  ds\right\vert  &  \leq CK_{2}\int_{x}^{+\infty}%
e^{\sigma\left(  x-s\right)  }\left\vert R\left(  s\right)  \right\vert
e^{\sigma s}ds\\
&  =CK_{2}e^{\sigma x}\int_{x}^{+\infty}\left\vert R\left(  s\right)
\right\vert ds.
\end{align*}
On the other hand, there exists $\delta,K_{1}>0$ such that $\left\vert
Y_{1}\left(  x\right)  \right\vert \leq K_{1}e^{\left(  \sigma-\delta\right)
x}$ for $x\geq0.$ Hence if $\left\vert \eta\left(  s\right)  \right\vert \leq
Ce^{\sigma s},$ then
\begin{align}
\left\vert \int_{a}^{x}Y_{1}\left(  x-s\right)  R\left(  s\right)  \eta\left(
s\right)  ds\right\vert  &  \leq CK_{1}\int_{a}^{x}e^{\left(  \sigma
-\delta\right)  \left(  x-s\right)  }\left\vert R\left(  s\right)  \right\vert
e^{\sigma s}ds\nonumber\\
&  \leq CK_{1}e^{\sigma x}\int_{a}^{x}e^{-\delta\left(  x-s\right)
}\left\vert R\left(  s\right)  \right\vert ds\label{esy1}\\
&  \leq CK_{1}e^{\sigma x}\int_{a}^{x}\left\vert R\left(  s\right)
\right\vert ds.\nonumber
\end{align}
It follows from these two estimates that if $a$ is chosen such that
\[
\left(  K_{1}+K_{2}\right)  \int_{a}^{+\infty}\left\vert R\left(  s\right)
\right\vert ds<1,
\]
then the sequence $\left\{  \eta_{l}\right\}  $ will converge to a solution
$\phi_{j}^{+}$ of the equation $\left(  \ref{ode}\right)  $ satisfying%
\[
\left\vert \phi_{j}^{+}\left(  x\right)  \right\vert \leq Ce^{\sigma x}\text{
for }x\text{ large. }%
\]
Note that by the decay estimate $\left(  \ref{decay}\right)  $ of the lump
type solution, we have
\begin{equation}
\int_{a}^{+\infty}\left\vert R\left(  s\right)  \right\vert ds\leq\frac
{C}{1+\left\vert y\right\vert }\left(  \frac{\pi}{2}-\arctan\frac
{a}{1+\left\vert y\right\vert }\right)  . \label{esR}%
\end{equation}
Hence such $a$ always exists.

Now since
\[
\phi_{j}^{+}=e^{\lambda_{j}x}\xi_{j}+\int_{a}^{x}Y_{1}\left(  x-s\right)
R\left(  s\right)  \phi_{j}^{+}\left(  s\right)  ds-\int_{x}^{+\infty}%
Y_{2}\left(  x-s\right)  R\left(  s\right)  \phi_{j}^{+}\left(  s\right)  ds.
\]
we then can use $\left(  \ref{esy1}\right)  $ to deduce
\[
\phi_{j}^{+}e^{-\lambda_{j}x}-\xi_{j}\rightarrow0,\text{ as }x\rightarrow
+\infty.
\]
Similar arguments yield the solution $\phi_{j}^{-}.$ The choice of $a$ for
$\phi_{j}^{+}$ and $\phi_{j}^{-}$ will be denoted by $a^{+}$ and $a^{-}$
respectively. This finishes the proof.
\end{proof}

Now without of generality we assume that $\operatorname{Re}\lambda
_{1}>\operatorname{Re}\lambda_{2}>\operatorname{Re}\lambda_{3}.$ The matrices
\begin{equation}
\Phi^{+}:=\left[  \phi_{1}^{+},\phi_{2}^{+},\phi_{3}^{+}\right]  ,\Phi
^{-}:=\left[  \phi_{1}^{-},\phi_{2}^{-},\phi_{3}^{-}\right]  . \label{cp}%
\end{equation}
are related by a matrix $W:$%
\[
\Phi^{+}=\Phi^{-}M.
\]
The matrix $M$ has a unique lower triangular-diagonal-upper triangular
factorization $M=\mathbf{L}\delta\mathbf{U}^{-1},$ where the entries of
$\mathbf{L},\mathbf{U}$ are equal to $1.$ We have $\Phi^{+}\mathbf{U}=\Phi
^{-}\mathbf{L}\delta.$ Moreover,
\begin{align*}
\Phi^{+}\mathbf{U}e^{-\lambda_{j}x}  &  \rightarrow\xi_{j},\text{ as
}x\rightarrow+\infty,\\
\Phi^{-}\mathbf{L}e^{-\lambda_{j}x}  &  \rightarrow\xi_{j},\text{ as
}x\rightarrow+\infty.
\end{align*}
Hence $\Phi^{+}\mathbf{U}$ is required Beals-Coifman fundamental solution matrix.

This solution is meromorphic in each $\Omega_{1,j},j=1,...,6.$ The restriction
of $U^{ou}$ to $\Omega_{1,j}$ will be denoted by $U_{j}^{ou}.$ On the common
boundaries of $\Omega_{1,j}$ and $\Omega_{1,j+1},$ $U_{j}^{ou}$ and
$U_{j+1}^{ou}$are related by the transfer matrix $V_{j}.$ That is, for
$j=1,...,6,$
\[
U_{j+1}^{ou}=U_{j}^{ou}V_{j}.
\]
Here we set $U_{7}^{ou}=U_{1}^{ou}.$

\begin{lemma}
\label{transfer}The transfer matrix $V_{k}$ is equal to the identity matrix
$I.$
\end{lemma}

\begin{proof}
In terms of the functions $\Phi^{+}$ and $\Phi^{-}$ defined in $\left(
\ref{cp}\right)  ,$ we can write
\[
U_{j}^{ou}=\Phi^{+}\mathbf{U=}\Phi^{-}\mathbf{L}\delta\mathbf{.}%
\]
We use the same notation with a tilt to denote the corresponding functions of
$U_{j+1}^{ou}.$ That is,
\[
U_{j+1}^{ou}=\tilde{\Phi}^{+}\mathbf{\tilde{U}=}\tilde{\Phi}\mathbf{^{-}%
\mathbf{\tilde{L}}}\tilde{\delta}\mathbf{.}%
\]
The matrix $V_{j}$ is independent of $x,y$. On the common boundaries of
$\Omega_{1,j}$ and $\Omega_{1,j+1},$ by $\left(  \ref{esR}\right)  ,$ we have,
for $k=1,2,3,$
\begin{align*}
\lim_{y\rightarrow+\infty}\left[  \left(  \phi_{k}^{+}-\tilde{\phi}_{k}%
^{+}\right)  e^{-\lambda_{k}x}\right]   &  =0,\text{ uniformly for }x>a^{+},\\
\lim_{y\rightarrow+\infty}\left[  \left(  \phi_{k}^{-}-\tilde{\phi}_{k}%
^{-}\right)  e^{-\lambda_{k}x}\right]   &  =0,\text{ uniformly for }x<a^{-}.
\end{align*}
From this we deduce
\[
\lim_{y\rightarrow+\infty}U_{j+1}^{ou}\left(  U_{j}^{ou}\right)  ^{-1}=I.
\]
Hence the transfer matrix $V_{j}$ equals identity.
\end{proof}

In $\Omega_{2},$ we define $U^{in}$, matrix solution of $\left(
\ref{firstorder}\right)  $, such that
\[
U^{in}\left(  0\right)  =I.
\]
A key property is that $U^{in}$ is holomorphic in $\Omega_{2}.$ In general,
assuming the jump matrix from the interior to the outer solutions at the
boundary circle $\partial B_{r}$ has the form
\[
G\left(  k\right)  E\left(  k\right)  .
\]
Then we have the following relation:
\begin{equation}
U^{ou}=U^{in}G\left(  k\right)  E\left(  k\right)  . \label{holom}%
\end{equation}
Taking $\left(  x,y\right)  =\left(  0,0\right)  ,$ we get
\[
U^{ou}\left(  0,0\right)  =G\left(  k\right)  E\left(  k\right)  .
\]
It turns out that the Beals-Coifman fundamental solution $U^{ou}\left(
k;x,y\right)  $ has the form
\[
U^{ou}\left(  k;x,y\right)  =\left[  I+%
%TCIMACRO{\dsum \limits_{k_{j}^{\ast}}}%
%BeginExpansion
{\displaystyle\sum\limits_{k_{j}^{\ast}}}
%EndExpansion%
%TCIMACRO{\dsum \limits_{s=1}^{n_{j}}}%
%BeginExpansion
{\displaystyle\sum\limits_{s=1}^{n_{j}}}
%EndExpansion
\left(  \frac{A_{j,s}\left(  x,y\right)  }{\left(  k-k_{j}^{\ast}\right)
^{s}}\right)  \right]  E\left(  k\right)  e^{M\left(  k\right)  x},
\]
for certain complex numbers $k_{j}^{\ast}.$ Here $A_{j,s}$ are $3\times3$
matrices. We then obtain
\begin{equation}
\left[  I+%
%TCIMACRO{\dsum \limits_{k_{j}^{\ast}}}%
%BeginExpansion
{\displaystyle\sum\limits_{k_{j}^{\ast}}}
%EndExpansion%
%TCIMACRO{\dsum \limits_{s=1}^{n_{j}}}%
%BeginExpansion
{\displaystyle\sum\limits_{s=1}^{n_{j}}}
%EndExpansion
\left(  \frac{A_{j,s}\left(  x,y\right)  }{\left(  k-k_{j}^{\ast}\right)
^{s}}\right)  \right]  E\left(  k\right)  e^{Mx}E\left(  k\right)
^{-1}G\left(  k\right)  ^{-1}=U^{in}. \label{holo}%
\end{equation}
As we already mentioned, $U^{in}$ is a holomorphic function in the radius $r$
disk. This will yield a system of equations for the entries of $A_{j,s}.$ Next
we would like to show that the system has a unique solution.

\begin{lemma}
\label{unique}For fixed $G,$ the system $\left(  \ref{holom}\right)  $ has a
unique solution.
\end{lemma}

\begin{proof}
We have
\[
U^{ou}\left(  k;x,y\right)  =U^{in}\left(  k;x,y\right)  G\left(  k\right)  .
\]
Suppose there is another pair $\left(  \tilde{U}^{ou},\tilde{U}^{in}\right)  $
such that
\[
\tilde{U}^{ou}\left(  k;x,y\right)  =\tilde{U}^{in}\left(  k;x,y\right)
G\left(  k\right)  .
\]

We claim that $\tilde{U}^{ou}=U^{ou}.$

Indeed, since $\tilde{U}^{ou}$ is invertible, the matrix $U^{ou}\left(
\tilde{U}^{ou}\right)  ^{-1}$ is holomorphic outside $B_{r}$, while
$U^{in}\left(  \tilde{U}^{in}\right)  ^{-1}$ is holomorphic inside $B_{r}.$
Moreover, they are equal to each other on $\partial B_{r}.$ Hence they patch
up to an entire holomorphic function which is also bounded. Hence in view of
their asymptotics at infinity, we obtain
\[
U^{ou}=\tilde{U}^{ou},\text{ and }U^{in}=\tilde{U}^{in}.
\]
This finishes the proof.
\end{proof}

With this result at hand, next we show that the solution $q$ has to be
rational. Let us define
\[
\Phi:=E\left(  k\right)  e^{Mx}E\left(  k\right)  ^{-1}.
\]

\begin{lemma}
The matrix $\Phi$ is holomorphic in $k$ with removable singulaities at
\[
k_{1,\pm}=\pm\frac{\sqrt{6}}{3}i,k_{2,\pm}=\pm\frac{2\sqrt{6}}{3}i.
\]

\end{lemma}

\begin{proof}
For instance, if $\Phi_{i,j}$ represents the entry of $\Phi$ on the $i$-th row
and $j$-th column, then%
\begin{align*}
\Phi_{12}  &  =\frac{ki\sqrt{3k^{2}+8}-3k^{2}-4}{\left(  6k^{2}+4\right)
\sqrt{3k^{2}+8}}e^{-\frac{x}{2}\left(  ki+\sqrt{3k^{2}+8}\right)  }-\frac
{ki}{3k^{2}+2}e^{kix}\\
&  +\frac{ki\sqrt{3k^{2}+8}+3k^{2}+4}{\left(  6k^{2}+4\right)  \sqrt{3k^{2}%
+8}}e^{-\frac{x}{2}\left(  ki-\sqrt{3k^{2}+8}\right)  }.
\end{align*}
Letting $t=\frac{x}{2}\sqrt{3k^{2}+8},$ we see that
\[
\Phi_{12}=\frac{ki}{3k^{2}+2}e^{-\frac{kxi}{2}}\cos t+\frac{3k^{2}+4}%
{6k^{2}+4}xe^{-\frac{kxi}{2}}\frac{\sin t}{t}-\frac{ki}{3k^{2}+2}e^{kix}.
\]
Note that $\cos t$ and $\sin t$ are holomorphic in $k,$ and the derivatives of
them with respect to $k$ contains polynomials of $x$ as coefficients.
Similarly, for $\Phi_{22}:$
\begin{align*}
\Phi_{22}  &  =\frac{\left(  k^{2}+1\right)  \sqrt{3k^{2}+8}-ki}{\left(
3k^{2}+2\right)  \sqrt{3k^{2}+8}}e^{-\frac{x}{2}\left(  ki+\sqrt{3k^{2}%
+8}\right)  }+\frac{k^{2}}{3k^{2}+2}e^{kix}\\
&  +\frac{\left(  k^{2}+1\right)  \sqrt{3k^{2}+8}+ki}{\left(  3k^{2}+2\right)
\sqrt{3k^{2}+8}}e^{-\frac{x}{2}\left(  ki-\sqrt{3k^{2}+8}\right)  }\\
&  =\frac{2k^{2}+2}{3k^{2}+2}e^{-\frac{kxi}{2}}\cos t+\frac{ki}{3k^{2}%
+2}xe^{-\frac{kxi}{2}}\frac{\sin t}{t}+\frac{k^{2}}{3k^{2}+2}e^{kix}.
\end{align*}
The other entries can be treated in a similar way.

Note that potentially $\Phi_{i,j}$ also has singularities when $k=\pm
\frac{\sqrt{6}}{3}.$ However, one can also show that they are removable.
\end{proof}

Note that the second equation in the Lax pair reads as
\begin{equation}
\partial_{y}f=i\left(  3\frac{d^{2}}{dx^{2}}-4\left(  q+1\right)  \right)
f.\label{equ:y}%
\end{equation}
Recall that we have defined $f_{1}=f$, and $\mathbf{f:=}\left(  f_{1}%
,f_{2},f_{3}\right)  ^{T},$ which solves $\left(  \ref{firstorder}\right)  ,$
the ODE system corresponding to the first equation of the Lax pair. In view of
the asymptotic behavior imposed on the Beals-Coifman solution, we have
$\mathbf{f}e^{-\lambda_{j}x}\rightarrow\xi_{j}$  as $x\rightarrow+\infty$. We
then see that the function
\[
e^{i\left(  3\lambda_{j}^{2}-4\right)  y}f_{1}\left(  x\right)
\]
will solve the equation $\left(  \ref{equ:y}\right)  .$ Let us set $\sigma
_{j}=i\left(  3\lambda_{j}^{2}-4\right)  ,$ and $\Lambda_{j}=\lambda
_{j}x+\sigma_{j}y.$ 

\begin{lemma}
There holds
\begin{align*}
\Lambda_{1}\left(  k_{1,+}\right)   &  =-\frac{\sqrt{6}}{3}x-2iy,\Lambda
_{2}\left(  k_{1,+}\right)  =\frac{2\sqrt{6}}{3}x+4iy,\Lambda_{3}\left(
k_{1,+}\right)  =-\frac{\sqrt{6}}{3}x-2iy,\\
\Lambda_{1}\left(  k_{1,-}\right)   &  =\frac{\sqrt{6}}{3}x-2iy,\Lambda
_{2}\left(  k_{1,-}\right)  =\frac{\sqrt{6}}{3}x+4iy,\Lambda_{3}\left(
k_{1,-}\right)  =-\frac{2\sqrt{6}}{3}x-2iy.
\end{align*}
Moreover,
\begin{align*}
\Lambda_{1}\left(  k_{2,+}\right)   &  =-\frac{2\sqrt{6}}{3}x+4iy,\Lambda
_{2}\left(  k_{1,+}\right)  =\frac{\sqrt{6}}{3}x-2iy,\Lambda_{3}\left(
k_{1,+}\right)  =\frac{\sqrt{6}}{3}x-2iy,\\
\Lambda_{1}\left(  k_{2,-}\right)   &  =\frac{2\sqrt{6}}{3}x+4iy,\Lambda
_{2}\left(  k_{1,-}\right)  =-\frac{\sqrt{6}}{3}x-2iy,\Lambda_{3}\left(
k_{1,-}\right)  =-\frac{\sqrt{6}}{3}x-2iy.
\end{align*}

\end{lemma}

\begin{proof}
This follows from direct computation.  
\end{proof}

\begin{lemma}
\label{poly}Suppose $Q$ is a rational function of the $x,y$ variables. Then
for each fixed $y,$
\[
\lim_{x\rightarrow\infty}\frac{\partial_{x}Q}{Q}=0.
\]

\end{lemma}

\begin{proof}
$Q$ can be written as $\frac{V}{W},$ where $W,V$ are polynomials. For fixed
$y,$ without loss of generality, we assume $V,W>0$ for $x$ large. Then
\begin{align*}
\frac{\partial_{x}Q}{Q} &  =\partial_{x}\ln Q=\partial_{x}\ln V-\partial
_{x}\ln W\\
&  =\frac{\partial_{x}V}{V}-\frac{\partial_{x}W}{W}.
\end{align*}
Since $V,W$ are polynomials, we conclude
\[
\lim_{x\rightarrow\infty}\frac{\partial_{x}Q}{Q}=0.
\]
The proof is then completed.
\end{proof}

Now we are ready to prove the main result of this section.

\begin{theorem}
Suppose $q$ is a solution of the Boussinesq equation satisfying the assumption
of Theorem, then $q$ is rational.
\end{theorem}

\begin{proof}
Once the solution $q$ is given, the matrix $G$ is determined. On the other
hand, the solution $q,$ which appears as a potential in the Lax pair equation,
is determined by the Beals-Coifman function $U^{ou}$ and $U^{in}.$ Hence we
need to determine the matrices $A_{j,s}$ in $\left(  \ref{holo}\right)  .$

We first show that the possible poles $k_{j}^{\ast}$ in $\left(
\ref{holo}\right)  $ has to be $k_{1,\pm}$ and $k_{2,\pm}.$ To see this, we
use the fact that the Beals-Coifman fundamental solution is unique. Then for
$y$ large, the constants $a^{\pm}$ appeared in the construction of
Beals-Coifman fundamental solution can both be chosen to be zero. Note that
this construction works provided that $\lambda_{j}$ are distinct. Hence again
using the estimate $\left(  \ref{esR}\right)  ,$ we then see that as $y$ tends
to $\infty,$ the asymptotic behavior of the solution $\phi_{j}^{+}$ is
\textquotedblleft close\textquotedblright\ to the asymptotic behavior of
$\phi_{j}^{-}$ for $x$ large. Hence if a complex number $k^{\ast}$ is not
equal to $k_{1,\pm}$ or $k_{2,\pm},$ then it can not appear in the set of poles.

Since the right hand side of $\left(  \ref{holo}\right)  $ is holomorphic in
$k,$ we see that the matrices $A_{j,s}$ satisfy a system of linear equation
whose entries are polynomial in $x.$ Now by Lemma \ref{unique}, the solution
has to be unique. Hence the linear system does not have kernel and $q$
contains rational functions and exponential functions in its expression.

We claim that $q$ is rational. Indeed, supposes $q$ also have exponential
functions, then we can write
\[
q\left(  x,y\right)  =%
%TCIMACRO{\dsum \limits_{k,j=0}^{+\infty}}%
%BeginExpansion
{\displaystyle\sum\limits_{k,j=0}^{+\infty}}
%EndExpansion
\left[  Q_{j,k}\left(  x,y\right)  e^{-\sqrt{\frac{2}{3}}jx-2kyi}\right]  .
\]
Inserting it into the equation
\[
KP\left(  q\right)  :=3\partial_{x}^{2}\left(  \partial_{x}^{2}q-4q^{2}%
-8q\right)  -q_{yy}=0,
\]
we see that $KP\left(  Q_{0,0}\right)  =0,$ and%
\begin{align*}
&  3\partial_{x}^{2}\left[  \partial_{x}^{2}\left(  Q_{1,0}e^{-\sqrt{\frac
{2}{3}}x}\right)  -8Q_{0,0}Q_{1,0}e^{-\sqrt{\frac{2}{3}}x}-8\left(
Q_{1,0}e^{-\sqrt{\frac{2}{3}}x}\right)  \right] \\
&  -\partial_{y}^{2}\left(  Q_{1,0}e^{-\sqrt{\frac{2}{3}}x}\right) \\
&  =0.
\end{align*}

Let us set $a=-\sqrt{\frac{2}{3}}.$ Then the left hand side can be written as
\[
\left(  3a^{4}-8a^{2}+P\left(  x,y\right)  \right)  Q_{1,0}=0,
\]
where $P$ is determined by $Q_{0,0}$ and derivatives of $Q_{1,0}.$ In
particular, applying Lemma \ref{poly}, we have
\[
P\left(  x,y\right)  \rightarrow0,\text{ as }x\rightarrow+\infty.
\]
It follows that
\[
Q_{1,0}=0.
\]
Now for general $Q_{j,k},$ it satisfies an equation of the form
\[
\left[  3\left(  \left(  aj\right)  ^{4}-8\left(  aj\right)  ^{2}\right)
-\left(  2ki\right)  ^{2}+P_{j,k}\right]  Q_{j,k}=0.
\]
Observe that
\begin{align*}
&  3\left(  \left(  aj\right)  ^{4}-8\left(  aj\right)  ^{2}\right)  -\left(
2ki\right)  ^{2}\\
&  =\frac{4}{3}\left(  j^{4}-12j^{2}+3k^{2}\right)  .
\end{align*}
This is nonzero for all integers $j,k.$ Hence same arguments as above implies
that $Q_{j,k}=0$ for $k+j\geq1.$ We then conclude that $q$ is a rational solution.

It is worth pointing out that by the Krichever theorem(see
\cite{Krichever78,Krichever83}), if the solution is rational in $x,$ then it
will also be rational in $y.$
\end{proof}

\section{The Boussinesq hierarchy and the structure of rational solutions}

In this section, we will extend the techniques developed in \cite{Moser} for
the KdV equation, to classify the degree of the tau functions of the rational
solutions of the Boussinesq equation. This problem is originally raised in P.
123--P. 124 of \cite{Moser}. It turns out to be much more delicate than the
KdV case.

\subsection{The Boussinesq hierarchy}

In \cite{Mckean}, Mckean found the Boussinesq hierarchy associated to the
Boussinesq equation. The equation he studied is the following:
\begin{equation}
\partial_{y}^{2}q=3\partial_{x}^{2}\left(  \partial_{x}^{2}q+4q^{2}\right)
.\label{eq:Bou}%
\end{equation}
Related works on Boussinesq hierarchy can be found in \cite{G}.

We use $D$ to denote the differentiation with respect to the $x$ variable(not
the bilinear derivative operator). Define the operator
\[
\mathcal{D}=
\begin{bmatrix}
0 & D\\
D & 0
\end{bmatrix}
.
\]
Let
\begin{align*}
L_{0}  &  :=D^{5}+5\left(  qD^{3}+D^{3}q\right) \\
&  -3\left(  q^{\prime\prime}D+Dq^{\prime\prime}\right)  +16qDq,
\end{align*}
and define
\[
\mathcal{K}_{0}=
\begin{bmatrix}
D^{3}+qD+Dq & 3pD+2p^{\prime}\\
3pD+p^{\prime} & \frac{1}{3}L_{0}%
\end{bmatrix}
.
\]
He then uses $\mathcal{K}_{0}$ to define recursively a sequence of vector
fields, which generate the Boussinesq hierarchy.

In our case, we are actually considering those solutions of the Boussinesq
equation $\left(  \ref{eq:Bou}\right)  $ with nonzero boundary condition, say
$\tilde{q}\rightarrow-\frac{1}{8}$ as $x^{2}+y^{2}\rightarrow+\infty.$ Indeed,
introducing new variable $\tilde{q}$ by $q=\tilde{q}-\frac{1}{8}$ in $\left(
\ref{eq:Bou}\right)  ,$ we obtain
\begin{equation}
\partial_{y}^{2}\tilde{q}=3\partial_{x}^{2}\left(  \partial_{x}^{2}\tilde
{q}+4\tilde{q}^{2}-\tilde{q}\right)  . \label{eq:B}%
\end{equation}
If we set $\tilde{q}\left(  x,y\right)  =\frac{3}{4}u\left(  x,\sqrt
{3}y\right)  ,$ then $u$ satisfies the version of Boussinesq equation appeared
in the first section, that is:
\[
\partial_{x}^{2}\left(  \partial_{x}^{2}u+3u^{2}-u\right)  -\partial_{y}
^{2}u=0.
\]

We are thus lead to consider the shifted operator $L$ defined by
\begin{align*}
L  &  :=D^{5}+5\left[  \left(  q-\frac{1}{8}\right)  D^{3}+D^{3}\left(
q-\frac{1}{8}\right)  \right] \\
&  -3\left(  q^{\prime\prime}D+Dq^{\prime\prime}\right)  +16\left(  q-\frac
{1}{8}\right)  D\left(  q-\frac{1}{8}\right)  .
\end{align*}
Note that
\[
L=L_{0}-\frac{5}{4}D^{3}-2\left(  qD+Dq\right)  +\frac{1}{4}D
\]
We then define
\[
\mathcal{K}_{i}=
\begin{bmatrix}
D^{3}+\left(  q-\frac{1}{8}\right)  D+D\left(  q-\frac{1}{8}\right)  &
3\left(  p+\left(  -1\right)  ^{i}a\right)  D+2p^{\prime}\\
3\left(  p+\left(  -1\right)  ^{i}a\right)  D+p^{\prime} & \frac{1}{3}L
\end{bmatrix}
.
\]
Here the constant $a$ is chosen such that
\[
\left(  3a\right)  ^{2}+\frac{1}{48}=0.
\]
We will explain later on why $a$ should be chosen in this way. Let $H_{0}%
=\int\frac{3p}{2}.$ Then a series of vector fields can be defined recursively
by
\[
X_{n+1}=\mathcal{K}_{n}\mathcal{\nabla}H_{n},\text{ and }\mathcal{D}\nabla
H_{n}=X_{n}.
\]
More precisely, once we obtained $X_{j},$ we can find $\nabla H_{j}$ by using
the relation $\mathcal{D}\nabla H_{j}=X_{j}.$ Then we can find $X_{j+1}$ by
$X_{j+1}=\mathcal{K\nabla}H_{j}.$

In particular,
\[
X_{0}=\mathcal{D}\nabla H_{0}=0,X_{1}=\mathcal{K}\nabla H_{0}=\left(
3p^{\prime},q^{\prime\prime\prime}+8qq^{\prime}-q^{\prime}\right)  .
\]
Hence $X_{2}$ is the Boussinesq flow. Here $^{\prime}$ represents the
derivative with respect to the $x$ variable.

For real valued solutions, as we will see, the main order term of the $\tau$
function is $\left(  x^{2}+3y^{2}\right)  ^{n}.$

Suppose $u$ is a rational solution of the KP-I equation. Then from
\cite{Krichever78,Krichever83}, we know that $u$ can be written in the form
\[
u=-\frac{3}{2}\sum_{i=1}^{n}\frac{1}{\left(  x-\xi_{i}\left(  y,t\right)
\right)  ^{2}}.
\]
In this case, $u=\frac{3}{2}\partial_{x}^{2}\ln\tau,$ where $\tau$ is a
polynomial in the $x$ variable.

For rational solutions $q$ of the Boussinesq equation, we have
\begin{equation}
q=-\frac{3}{2}\sum_{i=1}^{n}\frac{1}{\left(  x-\eta_{i}\left(  y\right)
\right)  ^{2}}. \label{qBou}%
\end{equation}
inserting this into the equation $\left(  \ref{eq:B}\right)  ,$ we find that
for each fixed index $i=1,...,n,$ there holds
\begin{equation}
\partial_{y}^{2}\eta_{i}-\frac{72}{\left(  \eta_{i}-\eta_{j}\right)  ^{3}}=0,
\label{Cal}%
\end{equation}%
\[
\eta_{i}^{\prime2}+36\Sigma_{j\neq i}\left(  \eta_{j}-\eta_{i}\right)
^{-2}+3=0.
\]

Note that $\left(  \ref{Cal}\right)  $ is the Caloger-Moser system. More
precisely, let $\partial_{y}\eta_{i}=\beta_{i}.$ Then the CM flow is
\begin{equation}
\left\{
\begin{array}
[c]{l}%
\partial_{y}\eta_{i}=\beta_{i},\\
\partial_{y}\beta_{i}=\frac{72}{\left(  \eta_{i}-\eta_{j}\right)  ^{3}}.
\end{array}
\right.  \label{cm}%
\end{equation}
Indeed, one can show that the function $\left(  \ref{qBou}\right)  $ solves
the Boussinesq equation if and only $\left(  \eta,\beta\right)  $ satisfies
the CM system $\left(  \ref{cm}\right)  $ restricted to the set
\[
M:=\left\{  \left(  \eta,\beta\right)  \in\mathbb{C}^{2n}:\nabla\left(
F_{3}+F_{1}\right)  =0\right\}  ,
\]
where
\begin{align*}
F_{1}  &  =3%
%TCIMACRO{\dsum \limits_{j=1}^{n}}%
%BeginExpansion
{\displaystyle\sum\limits_{j=1}^{n}}
%EndExpansion
\beta_{j},\text{ }\\
F_{3}  &  =\frac{1}{3}%
%TCIMACRO{\dsum \limits_{j=1}^{n}}%
%BeginExpansion
{\displaystyle\sum\limits_{j=1}^{n}}
%EndExpansion
\beta_{j}^{3}+36%
%TCIMACRO{\dsum \limits_{j=1}^{n}}%
%BeginExpansion
{\displaystyle\sum\limits_{j=1}^{n}}
%EndExpansion%
%TCIMACRO{\dsum \limits_{k\neq j}}%
%BeginExpansion
{\displaystyle\sum\limits_{k\neq j}}
%EndExpansion
\frac{\beta_{j}}{\left(  \eta_{j}-\eta_{k}\right)  ^{2}}.
\end{align*}
The proof follows from similar lines as that of \cite{Moser}, although the
case of hyperbolic Boussinesq equation is treated there, instead of the
elliptic case we are studying now. Therefore we omit the details.

Explicitly, a point $\left(  \eta_{1},...,\eta_{n},\beta_{1},...,\beta
_{n}\right)  \in M\subset\mathbb{C}^{2n}$ if and only if for each fixed
$j=1,...,n,$ the following identities hold:
\begin{equation}
\left\{
\begin{array}
[c]{l}%
%TCIMACRO{\dsum \limits_{k\neq j}}%
%BeginExpansion
{\displaystyle\sum\limits_{k\neq j}}
%EndExpansion
\frac{\beta_{j}+\beta_{k}}{\left(  \eta_{j}-\eta_{k}\right)  ^{3}}=0,\\
\beta_{j}^{2}+%
%TCIMACRO{\dsum \limits_{k\neq j}}%
%BeginExpansion
{\displaystyle\sum\limits_{k\neq j}}
%EndExpansion
\frac{36}{\left(  \eta_{j}-\eta_{k}\right)  ^{2}}+3=0.
\end{array}
\right.  \label{LocusM}%
\end{equation}
As a consequence, a vector $\left(  a_{1},...,a_{n},b_{1},...,b_{n}\right)
\in TM,$ the tangent space of $M$ at $\left(  \eta_{1},...,\eta_{n},\beta
_{1},...,\beta_{n}\right)  ,$ if and only if for each fixed $j=1,...,n,$
\begin{equation}
\left\{
\begin{array}
[c]{l}%
%TCIMACRO{\dsum \limits_{k\neq j}}%
%BeginExpansion
{\displaystyle\sum\limits_{k\neq j}}
%EndExpansion
\left(  \frac{b_{j}+b_{k}}{\left(  \eta_{j}-\eta_{k}\right)  ^{3}}%
-3\frac{\left(  \beta_{j}+\beta_{k}\right)  \left(  a_{j}-a_{k}\right)
}{\left(  \eta_{j}-\eta_{k}\right)  ^{4}}\right)  =0,\\
\beta_{j}b_{j}-%
%TCIMACRO{\dsum \limits_{k\neq j}}%
%BeginExpansion
{\displaystyle\sum\limits_{k\neq j}}
%EndExpansion
\frac{36\left(  a_{j}-a_{k}\right)  }{\left(  \eta_{j}-\eta_{k}\right)  ^{3}%
}=0.
\end{array}
\right.  \label{tangent}%
\end{equation}

Recall that the Boussinesq equation reads as
\[
\left\{
\begin{array}
[c]{l}%
q_{y}=3p^{\prime},\\
p_{y}=q^{\prime\prime\prime}+8qq^{\prime}-q^{\prime}.
\end{array}
\right.
\]
The rational solution $q$ of the Boussinesq equation can be written as
\[
q=-\frac{3}{2}%
%TCIMACRO{\dsum _{j=1}^{n}}%
%BeginExpansion
{\displaystyle\sum_{j=1}^{n}}
%EndExpansion
\frac{1}{\left(  x-\eta_{j}\right)  ^{2}}.
\]
Therefore, for this $q,$ we have%
\[
p=\frac{1}{2}%
%TCIMACRO{\dsum _{j=1}^{n}}%
%BeginExpansion
{\displaystyle\sum_{j=1}^{n}}
%EndExpansion
\frac{\beta_{j}}{\left(  x-\eta_{j}\right)  ^{2}}.
\]
where $\beta_{j}=\partial_{y}\eta_{j}.$ Now for initial condition $\left(
q_{0},p_{0}\right)  $ of this form, for each $k,$ the vector field $X_{k}$
corresponds the $k$-th Boussinesq flow can defined by
\[
\left(  q_{y},p_{y}\right)  ^{T}=X_{k}\left(  \left(  q,p\right)  ^{T}\right)
.
\]
Denote this flow by $e\left(  yX_{k}\right)  .$

\begin{proposition}
The $k$-th Boussinesq flow $e\left(  yX_{k}\right)  $ induces a flow on $M.$
More precisely,
\begin{equation}
X_{k}\left(  q,p\right)  =\left(  6%
%TCIMACRO{\dsum _{j=1}^{n}}%
%BeginExpansion
{\displaystyle\sum_{j=1}^{n}}
%EndExpansion
\frac{a_{j}}{\left(  x-\eta_{j}\right)  ^{3}},-%
%TCIMACRO{\dsum _{j=1}^{n}}%
%BeginExpansion
{\displaystyle\sum_{j=1}^{n}}
%EndExpansion
\left(  \frac{2\beta_{j}a_{j}}{\left(  x-\eta_{j}\right)  ^{3}}+\frac{b_{j}%
}{\left(  x-\eta_{j}\right)  ^{2}}\right)  \right)  , \label{Bouvec}%
\end{equation}
where $\left(  a_{1},...,a_{n},b_{1},...,b_{n}\right)  \in TM.$
\end{proposition}

\begin{proof}
We would like to write down the explicit form of each vector field $X_{k}$
acting on $\left(  q,p\right)  ,$ in an inductive way. We use the recursive
formula of $X_{i}$ defined through the operator $\mathcal{K}.$

Let us set
\[
m_{j}\left(  x\right)  :=\frac{1}{\left(  x-\eta_{j}\right)  ^{2}}%
,n_{j}\left(  x\right)  =\frac{1}{x-\eta_{j}}.
\]
We also use the notation $\Sigma_{j}=%
%TCIMACRO{\dsum _{j=1}^{n}}%
%BeginExpansion
{\displaystyle\sum_{j=1}^{n}}
%EndExpansion
.$ Then we have
\[
q=-\frac{3}{2}\Sigma_{j}m_{j},p=\frac{1}{2}\Sigma_{j}\left(  \beta_{j}%
m_{j}\right)  .
\]

For this vector, we have
\begin{align*}
\nabla H_{i}  &  =
\begin{bmatrix}
0 & D^{-1}\\
D^{-1} & 0
\end{bmatrix}
\left(  6\Sigma_{j}\frac{a_{j}}{\left(  x-\eta_{j}\right)  ^{3}},-\Sigma
_{j}\left(  \frac{2\beta_{j}a_{j}}{\left(  x-\eta_{j}\right)  ^{3}}
+\frac{b_{j}}{\left(  x-\eta_{j}\right)  ^{2}}\right)  \right)  ^{T}\\
&  =\left(  \left(  \Sigma_{j}\left(  \beta_{j}a_{j}m_{j}+b_{j}n_{j}\right)
\right)  ,-3\Sigma_{j}\left(  a_{j}m_{j}\right)  \right)  .
\end{align*}

Let us denote the first component of $\mathcal{K}X_{k}$ by $\left(
\mathcal{K}X_{k}\right)  ^{\left(  1\right)  }.$ Then assuming $\mathcal{K}
X_{k-1}$ has the form $\left(  \ref{Bouvec}\right)  $, we find that $\left(
\mathcal{K}X_{k}\right)  ^{\left(  1\right)  }$ equals
\begin{align*}
&  \left(  D^{3}+\left(  q-\frac{1}{8}\right)  D+D\left(  q-\frac{1}
{8}\right)  \right)  \left(  \Sigma_{j}\left(  \beta_{j}a_{j}m_{j}+b_{j}
n_{j}\right)  \right) \\
&  +\left(  3\left(  p+\alpha\right)  D+2p^{\prime}\right)  \left(
-3\Sigma_{j}\left(  a_{j}m_{j}\right)  \right)  .
\end{align*}
The points $\eta_{j},j=1,...,$ are possible poles. To analyze this function,
we would like to expand it around each pole $\eta_{j}$.

Let us fix an index $j.$ The coefficient before $\frac{1}{\left(  x-\eta
_{j}\right)  ^{5}}$ is
\begin{align*}
&  \left(  -24+2\left(  -\frac{3}{2}\right)  \left(  -2\right)  +\left(
-\frac{3}{2}\right)  \left(  -2\right)  \right)  \beta_{j}a_{j}\\
&  +\left(  -3\right)  \left(  3\left(  \frac{1}{2}\right)  \left(  -2\right)
+2\left(  \frac{1}{2}\right)  \left(  -2\right)  \right)  \beta_{j}a_{j}\\
&  =0.
\end{align*}
Therefore, $\left(  \mathcal{K\nabla}H_{k}\right)  ^{\left(  1\right)  }$ does
not have pole of order $5.$

Next we consider the term $\frac{1}{\left(  x-\eta_{j}\right)  ^{4}}.$ We see
that it only comes from
\[
\left(  D^{3}+qD+Dq\right)  \left(  b_{j}n_{j}\right)  .
\]
The coefficient vanishes, due to the fact that
\[
b_{j}\left(  \left(  -1\right)  \left(  -2\right)  \left(  -3\right)  +\left(
-\frac{3}{2}\right)  \left(  -1\right)  +\left(  -\frac{3}{2}\right)  \left(
-3\right)  \right)  =0.
\]

For the term $\frac{1}{\left(  x-\eta_{j}\right)  ^{3}},$ it comes from
\begin{align*}
&  \left(  2q-\frac{1}{4}\right)  \Sigma_{j}\left(  \beta_{j}a_{j}%
m_{j}^{\prime}\right)  +q^{\prime}\Sigma_{j}\left(  \beta_{j}a_{j}m_{j}\right)
\\
&  +q^{\prime}\Sigma_{j}\left(  b_{j}n_{j}\right)  -9\left(  p+\alpha\right)
\Sigma_{j}\left(  a_{j}m_{j}^{\prime}\right)  -6p^{\prime}\Sigma_{j}\left(
a_{j}m_{j}\right)  .
\end{align*}
Let us use $\Sigma_{k}^{\prime}$ to denote the summation over the index $k$
which is not equal to $j.$ The coefficient $I_{3}$ equals
\begin{align*}
&  2\left(  -\frac{3}{2}\right)  \Sigma_{k}^{\prime}\left(  \left(  -2\right)
\beta_{j}a_{j}m_{k}\left(  \eta_{j}\right)  \right)  -\frac{1}{4}\left(
-2\right)  \beta_{j}a_{j}\\
&  +\left(  -\frac{3}{2}\right)  \left(  -2\right)  \Sigma_{k}^{\prime}\left(
\beta_{k}a_{k}m_{k}\left(  \eta_{j}\right)  +b_{k}n_{k}\left(  \eta
_{j}\right)  \right) \\
&  -9\left(  \frac{1}{2}\right)  \Sigma_{k}^{\prime}\left(  -2\beta_{k}%
a_{j}m_{k}\left(  \eta_{j}\right)  \right)  -9\alpha\left(  -2\right)
a_{j}-6\left(  \frac{1}{2}\right)  \left(  -2\right)  \Sigma_{k}^{\prime
}\left(  \beta_{j}a_{k}m_{k}\left(  \eta_{j}\right)  \right)  .
\end{align*}
That is
\begin{align*}
I_{3}  &  =6\Sigma_{k}^{\prime}\left(  \beta_{j}a_{j}m_{k}\left(  \eta
_{j}\right)  \right)  +\frac{1}{2}\beta_{j}a_{j}\\
&  +3\Sigma_{k}^{\prime}\beta_{k}a_{k}m_{k}\left(  \eta_{j}\right)
+3\Sigma_{k}^{\prime}\left(  b_{k}n_{k}\left(  \eta_{j}\right)  \right) \\
&  +9\Sigma_{k}^{\prime}\left(  \beta_{k}a_{j}m_{k}\left(  \eta_{j}\right)
\right)  +18\alpha a_{j}+6\Sigma_{k}^{\prime}\left(  \beta_{j}a_{k}%
m_{k}\left(  \eta_{j}\right)  \right)  .
\end{align*}

For the $\frac{1}{\left(  x-\eta_{j}\right)  ^{2}}$ term, it is related to
\begin{align*}
&  2q\Sigma_{j}\left(  \beta_{j}a_{j}m_{j}^{\prime}+b_{j}n_{j}^{\prime
}\right)  +q^{\prime}\Sigma_{j}\left(  \beta_{j}a_{j}m_{j}+b_{j}n_{j}\right)
-\frac{1}{4}\Sigma_{j}\left(  b_{j}n_{j}^{\prime}\right) \\
&  +3p\Sigma_{j}\left(  -3a_{j}m_{j}^{\prime}\right)  +2p^{\prime}\Sigma
_{j}\left(  -3a_{j}m_{j}\right)  .
\end{align*}
Using the formula of $q$ and $p,$ we can compute its coefficient $I_{2}:$
\begin{align*}
&  2\left(  -\frac{3}{2}\right)  \left[  \left(  -2\right)  \Sigma_{k}%
^{\prime}\left[  \beta_{j}a_{j}m_{k}^{\prime}\left(  \eta_{j}\right)  \right]
+\Sigma_{k}^{\prime}\left[  \beta_{k}a_{k}m_{k}^{\prime}\left(  \eta
_{j}\right)  \right]  \right] \\
&  +\left(  -\frac{3}{2}\right)  \left[  \Sigma_{k}^{\prime}\left[  \beta
_{j}a_{j}m_{k}^{\prime}\left(  \eta_{j}\right)  \right]  +\left(  -2\right)
\Sigma_{k}\left[  \beta_{k}a_{k}m_{k}^{\prime}\left(  \eta_{j}\right)
\right]  \right] \\
&  +2\left(  -\frac{3}{2}\right)  \left(  -1\right)  b_{j}\Sigma_{k}^{\prime
}\left(  m_{k}\left(  \eta_{j}\right)  \right)  +2\left(  -\frac{3}{2}\right)
\Sigma_{k}^{\prime}\left[  b_{k}n_{k}^{\prime}\left(  \eta_{j}\right)
\right]  +\left(  -\frac{3}{2}\right)  \left(  -2\right)  \Sigma_{k}^{\prime
}\left[  b_{k}n_{k}^{\prime}\left(  \eta_{j}\right)  \right]  -\frac{1}%
{4}\left(  -1\right)  b_{j}\\
&  -9\left(  \frac{1}{2}\right)  \left[  \left(  -2\right)  \Sigma_{k}%
^{\prime}\left[  \beta_{k}a_{j}m_{k}^{\prime}\left(  \eta_{j}\right)  \right]
+\Sigma_{k}^{\prime}\left[  \beta_{j}a_{k}m_{k}^{\prime}\left(  \eta
_{j}\right)  \right]  \right] \\
&  -6\left(  \frac{1}{2}\right)  \left[  \Sigma_{k}^{\prime}\left[  \beta
_{k}a_{j}m_{k}^{\prime}\left(  \eta_{j}\right)  \right]  +\left(  -2\right)
\Sigma_{k}^{\prime}\left[  \beta_{j}a_{k}m_{k}^{\prime}\left(  \eta
_{j}\right)  \right]  \right]  .
\end{align*}
Then $I_{2}$ equals
\begin{align*}
&  \frac{9}{2}\Sigma_{k}^{\prime}\left[  \beta_{j}a_{j}m_{k}^{\prime}\left(
\eta_{j}\right)  \right]  +6\Sigma_{k}^{\prime}\left[  \beta_{k}a_{j}%
m_{k}^{\prime}\left(  \eta_{j}\right)  \right]  +\frac{3}{2}\Sigma_{k}%
^{\prime}\left[  \beta_{j}a_{k}m_{k}^{\prime}\left(  \eta_{j}\right)  \right]
\\
&  +3b_{j}\Sigma_{k}^{\prime}\left(  m_{k}\left(  \eta_{j}\right)  \right)
+\frac{1}{4}b_{j}.
\end{align*}
Note that
\[
12\Sigma_{k}^{\prime}\left(  m_{k}\left(  \eta_{j}\right)  \right)
+1=\frac{1}{3}\beta_{j}^{2},
\]%
\[
\Sigma_{k}^{\prime}\left[  \beta_{k}m_{k}^{\prime}\left(  \eta_{j}\right)
\right]  =-\Sigma_{k}^{\prime}\left[  \beta_{j}m_{k}^{\prime}\left(  \eta
_{j}\right)  \right]  .
\]
It follows that
\[
I_{2}=\frac{3}{2}\Sigma_{k}^{\prime}\left[  \beta_{j}\left(  a_{k}%
-a_{j}\right)  m_{k}^{\prime}\left(  \eta_{j}\right)  \right]  +\frac{1}%
{12}b_{j}\beta_{j}^{2}.
\]
Since $\left(  a_{1},...,a_{N},b_{1},...,b_{N}\right)  $ is in the tangent
space of $M,$ we obtain $I_{2}=0.$

Next, we compute the $\frac{1}{\left(  x-\eta_{j}\right)  }$ term. It comes
from
\[
2q\Sigma_{j}\left(  \beta_{j}a_{j}m_{j}^{\prime}+b_{j}n_{j}^{\prime}\right)
+q^{\prime}\Sigma_{j}\left(  \beta_{j}a_{j}m_{j}+b_{j}n_{j}\right)
-9p\Sigma_{j}\left(  a_{j}m_{j}^{\prime}\right)  -6p^{\prime}\Sigma_{j}\left(
a_{j}m_{j}\right)  .
\]
The corresponding coefficient $I_{1}$ is
\begin{align*}
&  2\left(  -\frac{3}{2}\right)  \Sigma_{k}^{\prime}\left(  \beta_{k}%
a_{k}m_{k}^{\prime\prime}\left(  \eta_{j}\right)  +b_{k}n_{k}^{\prime\prime
}\left(  \eta_{j}\right)  \right)  +2\left(  -\frac{3}{2}\right)  \Sigma
_{k}^{\prime}\left(  \left(  -2\right)  \frac{1}{2}\beta_{j}a_{j}m_{k}%
^{\prime\prime}\left(  \eta_{j}\right)  \right)  +2\left(  -\frac{3}%
{2}\right)  \Sigma_{k}^{\prime}\left(  \left(  -1\right)  b_{j}m_{k}^{\prime
}\left(  \eta_{j}\right)  \right) \\
&  +\left(  -\frac{3}{2}\right)  \left(  -2\right)  \Sigma_{k}^{\prime}\left(
\frac{1}{2}\beta_{k}a_{k}m_{k}^{\prime\prime}\left(  \eta_{j}\right)
+\frac{1}{2}b_{k}n_{k}^{\prime\prime}\left(  \eta_{j}\right)  \right)
+\left(  -\frac{3}{2}\right)  \Sigma_{k}^{\prime}\left(  \beta_{j}a_{j}%
m_{k}^{\prime\prime}\left(  \eta_{j}\right)  \right)  +\left(  -\frac{3}%
{2}\right)  \Sigma_{k}^{\prime}\left(  b_{j}m_{k}^{\prime}\right) \\
&  -9\left(  \frac{1}{2}\right)  \left(  \Sigma_{k}^{\prime}\left(  \beta
_{j}a_{k}m_{k}^{\prime\prime}\left(  \eta_{j}\right)  \right)  +\Sigma
_{k}^{\prime}\left(  \left(  -2\right)  \frac{1}{2}\beta_{k}a_{j}m_{k}%
^{\prime\prime}\right)  \right) \\
&  -6\left(  \frac{1}{2}\right)  \left(  \left(  -2\right)  \Sigma_{k}%
^{\prime}\left(  \frac{1}{2}\beta_{j}a_{k}m_{k}^{\prime\prime}\left(  \eta
_{j}\right)  \right)  +\Sigma_{k}^{\prime}\beta_{k}a_{j}m_{k}^{\prime\prime
}\right)  .
\end{align*}
It follows that
\begin{align}
I_{1}  &  =\frac{3}{2}\Sigma_{k}^{\prime}\left(  \beta_{j}a_{j}m_{k}%
^{\prime\prime}\left(  \eta_{j}\right)  \right)  +\frac{3}{2}\Sigma
_{k}^{\prime}\left(  \beta_{k}a_{j}m_{k}^{\prime\prime}\left(  \eta
_{j}\right)  \right) \nonumber\\
&  -\frac{3}{2}\Sigma_{k}^{\prime}\left(  \beta_{j}a_{k}m_{k}^{\prime\prime
}\left(  \eta_{j}\right)  \right)  -\frac{3}{2}\Sigma_{k}^{\prime}\left(
\beta_{k}a_{k}m_{k}^{\prime\prime}\left(  \eta_{j}\right)  \right) \nonumber\\
&  +\frac{3}{2}\Sigma_{k}^{\prime}\left(  b_{j}m_{k}^{\prime}\left(  \eta
_{j}\right)  \right)  +\frac{3}{2}\Sigma_{k}^{\prime}\left(  b_{k}%
m_{k}^{\prime}\left(  \eta_{j}\right)  \right)  . \label{CI1}%
\end{align}
Note that on $M,$ we have, for each fixed index $j,$
\[
\Sigma_{k}^{\prime}\left(  \left(  \beta_{k}+\beta_{j}\right)  m_{k}^{\prime
}\left(  \eta_{j}\right)  \right)  =0.
\]
As a consequence,
\[
\Sigma_{k}^{\prime}\left(  \left(  b_{k}+b_{j}\right)  m_{k}^{\prime}\left(
\eta_{j}\right)  \right)  +\Sigma_{k}^{\prime}\left(  \left(  \beta_{k}%
+\beta_{j}\right)  \left(  a_{j}-a_{k}\right)  m_{k}^{\prime\prime}\left(
\eta_{j}\right)  \right)  =0.
\]
This implies that $I_{1}=0.$

Now we consider the second component $\left(  \mathcal{K\nabla}H_{j}\right)
^{\left(  2\right)  }$ of the vector field $\mathcal{K\nabla}H_{j}.$ We have
\[
\left(  \mathcal{K\nabla}H_{j}\right)  ^{\left(  2\right)  }=-L\Sigma
_{j}\left(  a_{j}m_{j}\right)  +\left(  3\left(  p-\alpha\right)  D+p^{\prime
}\right)  \Sigma_{j}\left(  \beta_{j}a_{j}m_{j}+b_{j}n_{j}\right)  .
\]

Similar(but more tedious, the most complicated term is $16qDq$) computation as
above shows that the term $\frac{1}{\left(  x-\eta_{j}\right)  ^{k}}$ vanishes
for $k=1,4,5,6,7.$ Let us now compute the coefficient $J_{3}$ of $\frac
{1}{\left(  x-\eta_{j}\right)  ^{3}}.$ Recall that
\[
L=L_{0}-\frac{5}{4}D^{3}-2\left(  qD+Dq\right)  +\frac{1}{4}D
\]%
\[
L_{0}:=D^{5}+5\left(  qD^{3}+D^{3}q\right)  -3\left(  q^{\prime\prime
}D+Dq^{\prime\prime}\right)  +16qDq,
\]
Observe that $D^{3}\left(  q\Sigma_{j}\left(  a_{j}m_{j}\right)  \right)  $
does not contain $\frac{1}{\left(  x-\eta_{j}\right)  ^{3}}$ term. Hence from
the operator $L_{0},$ the contribution to the coefficient is:
\begin{align*}
&  5\left(  -\frac{3}{2}\right)  \left(  -24\right)  \Sigma_{k}^{\prime
}\left(  \frac{1}{2}a_{j}m_{k}^{\prime\prime}\left(  \eta_{j}\right)  \right)
\\
&  -3\left(  -\frac{3}{2}\right)  \left(  2\left(  6\right)  \Sigma
_{k}^{\prime}a_{k}m_{k}^{\prime\prime}\left(  \eta_{j}\right)  +2\left(
-2\right)  \Sigma_{k}^{\prime}\left(  a_{j}m_{k}^{\prime\prime}\left(
\eta_{j}\right)  \right)  +\left(  -24\right)  \Sigma_{k}^{\prime}\left(
\frac{1}{2}a_{k}m_{k}^{\prime\prime}\left(  \eta_{j}\right)  \right)  \right)
\\
&  +16\left(  \frac{9}{4}\right)  \left(  \Sigma_{k}^{\prime}\left(  \left(
-2\right)  a_{j}m_{k}^{\prime\prime}\left(  \eta_{j}\right)  \right)
+\Sigma_{k}^{\prime}\left(  m_{k}\left(  \eta_{j}\right)  \right)  \left(
-2\right)  \Sigma_{k}^{\prime}\left(  \left(  a_{k}+a_{j}\right)  m_{k}\left(
\eta_{j}\right)  \right)  \right)  .
\end{align*}
From the operator $-\frac{5}{4}D^{3}-2\left(  qD+Dq\right)  +\frac{1}{4}D,$ we
get
\[
-2\left(  -\frac{3}{2}\right)  \left(  2\left(  -2\right)  \Sigma_{k}^{\prime
}\left(  a_{j}m_{k}\left(  \eta_{j}\right)  \right)  +\left(  -2\right)
\Sigma_{k}^{\prime}\left(  a_{k}m_{k}\left(  \eta_{j}\right)  \right)
\right)  +\frac{1}{4}\left(  -2a_{j}\right)
\]
Finally, from
\[
\left(  3\left(  p+\alpha\right)  D+p^{\prime}\right)  \Sigma_{j}\left(
\beta_{j}a_{j}m_{j}+b_{j}n_{j}\right)  ,
\]
we obtain
\begin{align*}
&  3\left(  \frac{1}{2}\right)  \Sigma_{k}^{\prime}\left(  \left(  -2\right)
\beta_{k}\beta_{j}a_{j}m_{k}\left(  \eta_{j}\right)  \right)  +3\alpha\left(
-2\right)  \beta_{j}a_{j}\\
&  +\left(  \frac{1}{2}\right)  \left(  -2\right)  \Sigma_{k}^{\prime}\left(
\beta_{j}\beta_{k}a_{k}m_{k}\left(  \eta_{j}\right)  \right)  +\frac{1}%
{2}\left(  -2\right)  \Sigma_{k}^{\prime}\beta_{j}b_{k}n_{k}\left(  \eta
_{j}\right)  .
\end{align*}
Combining these, we obtain
\begin{align*}
J_{3}  &  =72\Sigma_{k}^{\prime}\left(  m_{k}\left(  \eta_{j}\right)  \right)
\Sigma_{k}^{\prime}\left(  \left(  a_{k}+a_{j}\right)  m_{k}\left(  \eta
_{j}\right)  \right) \\
&  +12\Sigma_{k}^{\prime}\left(  a_{j}m_{k}\left(  \eta_{j}\right)  \right)
+6\Sigma_{k}^{\prime}\left(  a_{k}m_{k}\left(  \eta_{j}\right)  \right)
+\frac{1}{2}a_{j}\\
&  -3\Sigma_{k}^{\prime}\left(  \beta_{k}\beta_{j}a_{j}m_{k}\left(  \eta
_{j}\right)  \right)  -\Sigma_{k}^{\prime}\left(  \beta_{j}\beta_{k}a_{k}%
m_{k}\left(  \eta_{j}\right)  \right) \\
&  -\Sigma_{k}^{\prime}\left(  \beta_{j}b_{k}n_{k}\left(  \eta_{j}\right)
\right)  -6\alpha\beta_{j}a_{j}.
\end{align*}
Now using the identity
\[
\beta_{j}^{2}+36\Sigma_{k}^{\prime}m_{k}\left(  \eta_{j}\right)  +3=0,
\]
we then see that $J_{3}=\beta_{j}I_{3}.$

The coefficient of $\frac{1}{\left(  x-\eta_{j}\right)  ^{2}}$ is also
nonzero, and can be computed in a similar way. However, to prove the assertion
of the proposition, it is not necessary to know its explicit formula.

In the sequel, for $j=1,...,n$, let us use $I_{j}$ to denote the coefficient
of degree $-3$ term for the pole $\eta_{j}$. We have now proved that
$\mathcal{K}\nabla H_{k}$ has the form
\[
\left(  6\Sigma_{j}\frac{I_{j}}{\left(  x-\eta_{j}\right)  ^{3}},-\Sigma
_{j}\left(  \frac{2\beta_{j}I_{j}}{\left(  x-\eta_{j}\right)  ^{3}}%
+\frac{B_{j}}{\left(  x-\eta_{j}\right)  ^{2}}\right)  \right)  =\mathcal{D}%
\nabla H_{k+1},
\]
Our next aim is to show that the vector
\[
\left(  I_{1},...,I_{n},B_{1},...,B_{n}\right)
\]
lies in the tangent space of $M$ at the point $\left(  \eta_{1},...,\eta
_{n},\beta_{1},...,\beta_{n}\right)  .$ To see this, it will be suffice to
show that $\mathcal{K}\nabla H_{k+1}$ is residue free at each pole, because
due to our previous computation, this means exactly the it is in the tangent
space of the locus $M$.

Let us write the operator $\mathcal{K}$ as
\[
\mathcal{K}=%
\begin{bmatrix}
K_{11} & K_{12}\\
K_{21} & K_{22}%
\end{bmatrix}
.
\]
We also write $\nabla H_{k+1}=\left(  \phi_{1},\phi_{2}\right)  ^{T}.$ That
is,
\[
\left(  \phi_{1},\phi_{2}\right)  =\left(  \Sigma_{j}\left(  I_{j}\beta
_{j}m_{j}+B_{j}n_{j}\right)  ,-3\Sigma_{j}\left(  I_{j}m_{j}\right)  \right)
.
\]
Introducing
\begin{equation}
\sigma=\Sigma_{j}\left(  a_{j}\beta_{j}m_{j}+b_{j}n_{j}\right)  ,\tau
=-3\Sigma_{j}\left(  a_{j}m_{j}\right)  , \label{sigmatau}%
\end{equation}
we get
\begin{align}
\phi_{1}^{\prime}  &  =K_{21}\sigma+K_{22}\tau,\label{fi}\\
\phi_{2}^{\prime}  &  =K_{11}\sigma+K_{12}\tau.\nonumber
\end{align}

Let $l$ be a closed path around the pole $\eta_{j}$ in the complex $x$ plane.
To see that the residue is zero(that is, does not have $\frac{1}{x-\eta_{j}}$
term in the Laurent expansion around $\eta_{j}$), we compute the integral
\begin{align*}
Q  &  :=\int_{l}\left(  \mathcal{K}\nabla H_{k+1}\right)  ^{T}dx\\
&  =\int_{l}\left[  K_{11}\phi_{1}+K_{12}\phi_{2},K_{21}\phi_{1}+K_{22}%
\phi_{2}\right]  dx.
\end{align*}
It is important to observe that each operator $K_{11},K_{22}$ is
skew-symmetric, and moreover the adjoint of $K_{12}$ is $-K_{21},$ that is,
\[
\int\left(  gK_{12}h\right)  =-\int\left(  hK_{21}g\right)  .
\]
This is to say that the matrix operator $\mathcal{K}$ is skew-symmetric.
Integrating by parts tells us that $Q$ equals
\[
-\int_{l}\left[  \phi_{1}K_{11}\left(  1\right)  +\phi_{2}K_{21}\left(
1\right)  ,\phi_{1}K_{12}\left(  1\right)  +\phi_{2}K_{22}\left(  1\right)
\right]  dx.
\]
Let us define $\ $
\[
\mu=%
\begin{bmatrix}
\mu_{1}\\
\mu_{2}%
\end{bmatrix}
:=\mathcal{K}%
\begin{bmatrix}
1\\
0
\end{bmatrix}
\text{ and }v=%
\begin{bmatrix}
v_{1}\\
v_{2}%
\end{bmatrix}
:=\mathcal{K}%
\begin{bmatrix}
0\\
1
\end{bmatrix}
.
\]
Then for some functions $w,s,$ we have
\[
\mu=\mathcal{D}\left(  w_{1},w_{2}\right)  ^{T}\text{ and }v=\mathcal{D}%
\left(  z_{1},z_{2}\right)  ^{T}.
\]
Explicitly,
\[
\mu=\left(  q^{\prime},p^{\prime}\right)  ^{T},\text{ }v=\left(  2p^{\prime
},\frac{1}{3}\left(  2q^{\prime\prime\prime}+16qq^{\prime}-2q^{\prime}\right)
\right)  ^{T}.
\]
With these notations,
\begin{align*}
Q  &  =-\int_{l}\left[  \phi_{1}\mu_{1}+\phi_{2}\mu_{2},\phi_{1}v_{1}+\phi
_{2}v_{2}\right]  dx\\
&  =-\int_{l}\left[  \phi_{1}w_{2}^{\prime}+\phi_{2}w_{1}^{\prime},\phi
_{1}z_{2}^{\prime}+\phi_{2}z_{1}^{\prime}\right]  dx\\
&  =\int_{l}\left[  \phi_{1}^{\prime}w_{2}+\phi_{2}^{\prime}w_{1},\phi
_{1}^{\prime}z_{2}+\phi_{2}^{\prime}z_{1}\right]  dx.
\end{align*}
Using $\left(  \ref{fi}\right)  ,$ we find that $Q$ is equal to
\begin{align*}
&  \int_{l}\left[  \left(  K_{21}\sigma+K_{22}\tau\right)  w_{2}+\left(
K_{11}\sigma+K_{12}\tau\right)  w_{1},\left(  K_{21}\sigma+K_{22}\tau\right)
z_{2}+\left(  K_{11}\sigma+K_{12}\tau\right)  z_{1}\right]  dx\\
&  =-\int_{l}\left[  \left(  K_{11}w_{1}+K_{12}w_{2}\right)  \sigma+\left(
K_{21}w_{1}+K_{22}w_{2}\right)  \tau,\left(  K_{11}z_{1}+K_{12}z_{2}\right)
\sigma+\left(  K_{21}z_{1}+K_{22}z_{2}\right)  \tau\right]  dx.
\end{align*}
In the case of
\[
\sigma=w_{1},\tau=w_{2},
\]
we have, using integration by parts,
\begin{align*}
&  \int_{l}\left[  \left(  K_{11}w_{1}+K_{12}w_{2}\right)  \sigma+\left(
K_{21}w_{1}+K_{22}w_{2}\right)  \tau\right]  dx\\
&  =\int_{l}\left[  \left(  K_{11}w_{1}+K_{12}w_{2}\right)  w_{1}+\left(
K_{21}w_{1}+K_{22}w_{2}\right)  w_{2}\right]  dx\\
&  =-\int_{l}\left[  \left(  K_{11}w_{1}+K_{12}w_{2}\right)  w_{1}+\left(
K_{21}w_{1}+K_{22}w_{2}\right)  w_{2}\right]  dx.
\end{align*}
This implies
\begin{equation}
\int_{l}\left[  \left(  K_{11}w_{1}+K_{12}w_{2}\right)  w_{1}+\left(
K_{21}w_{1}+K_{22}w_{2}\right)  w_{2}\right]  dx=0. \label{zero}%
\end{equation}
Therefore, if we write
\[
\mathcal{K}%
\begin{bmatrix}
w_{1}\\
w_{2}%
\end{bmatrix}
=\left(  -3\Sigma_{j}\left(  s_{j}m_{j}^{\prime}\right)  ,\Sigma_{j}\left(
s_{j}\beta_{j}m_{j}^{\prime}+t_{j}n_{j}^{\prime}\right)  \right)  ^{T},
\]
then in view of $\left(  \ref{CI1}\right)  ,$ $\left(  s_{1},...,s_{n}%
,t_{1},...,t_{n}\right)  $ satisfies the first equation of $\left(
\ref{tangent}\right)  .$ That is,
\begin{equation}
\Sigma_{k}^{\prime}\left(  \frac{t_{j}+t_{k}}{\left(  \eta_{j}-\eta
_{k}\right)  ^{3}}-\frac{3\left(  \beta_{j}+\beta_{k}\right)  \left(
s_{j}-s_{k}\right)  }{\left(  \eta_{j}-\eta_{k}\right)  ^{4}}\right)
=0,j=1,...,n. \label{stj}%
\end{equation}
We would like to show that $\left(  s_{1},...,s_{n},t_{1},...,t_{n}\right)  $
also satisfies the second equation of $\left(  \ref{tangent}\right)  .$ To do
this, for $\sigma,\tau$ with the form $\left(  \ref{sigmatau}\right)  ,$we
compute the integral along the closed circle $l$ which surrounds the $j$-th
pole $x_{j}.$ We have
\begin{align*}
&  T:=\int_{l}\left(  \left(  K_{11}w_{1}+K_{12}w_{2}\right)  \sigma+\left(
K_{21}w_{1}+K_{22}w_{2}\right)  \tau\right)  dx\\
&  =-3\int_{l}\left[  \Sigma_{k}\left(  s_{k}m_{k}^{\prime}\right)
\Sigma_{\mu}\left(  a_{\mu}\beta_{\mu}m_{\mu}+b_{\mu}n_{\mu}\right)
+\Sigma_{\mu}\left(  s_{\mu}\beta_{\mu}m_{\mu}^{\prime}+t_{\mu}n_{\mu}%
^{\prime}\right)  \Sigma_{k}\left(  a_{k}m_{k}\right)  \right]  dx\\
&  =-3\Sigma_{k}^{\prime}\left[  \left(  -1\right)  s_{j}a_{k}\beta_{k}%
m_{k}^{\prime\prime}\left(  \eta_{j}\right)  \right]  -3\Sigma_{k}^{\prime
}\left[  \left(  -1\right)  s_{j}b_{k}n_{k}^{\prime\prime}\right] \\
&  -3\Sigma_{k}^{\prime}\left[  s_{k}m_{k}^{\prime\prime}\left(  \eta
_{j}\right)  a_{j}\beta_{j}\right]  -3\Sigma_{k}^{\prime}\left[  s_{k}%
m_{k}^{\prime}\left(  \eta_{j}\right)  b_{j}\right] \\
&  -3\Sigma_{k}^{\prime}\left[  \left(  -1\right)  s_{j}\beta_{j}a_{k}%
m_{k}^{\prime\prime}\left(  \eta_{j}\right)  \right]  -3\Sigma_{k}^{\prime
}\left[  \left(  -1\right)  t_{j}a_{k}m_{k}^{\prime}\left(  \eta_{j}\right)
\right] \\
&  -3\Sigma_{k}^{\prime}\left[  a_{j}s_{k}\beta_{k}m_{k}^{\prime\prime}\left(
\eta_{j}\right)  \right]  -3\Sigma_{k}^{\prime}\left[  a_{j}t_{k}n_{k}%
^{\prime\prime}\right]  .
\end{align*}
Using $\left(  \ref{stj}\right)  ,$ we find that it equals%
\begin{align*}
&  3\Sigma_{k}^{\prime}\left[  \left(  \beta_{j}+\beta_{k}\right)  \left(
s_{j}-s_{k}\right)  a_{j}m_{k}^{\prime\prime}\left(  x_{j}\right)  \right]
-3\Sigma_{k}^{\prime}\left[  \left(  \beta_{j}+\beta_{k}\right)  \left(
a_{j}-a_{k}\right)  s_{j}m_{k}^{\prime\prime}\left(  x_{j}\right)  \right] \\
&  -3\Sigma_{k}^{\prime}\left[  \left(  s_{j}b_{k}+s_{k}b_{j}-t_{j}a_{k}%
-a_{j}t_{k}\right)  m_{k}^{\prime}\left(  x_{j}\right)  \right] \\
&  =3\Sigma_{k}^{\prime}\left[  \left(  a_{j}\left(  t_{j}+t_{k}\right)
-s_{j}\left(  b_{j}+b_{k}\right)  +s_{j}b_{k}+s_{k}b_{j}-t_{j}a_{k}-a_{j}%
t_{k}\right)  m_{k}^{\prime}\left(  x_{j}\right)  \right] \\
&  =3\Sigma_{k}^{\prime}\left[  \left(  a_{j}-a_{k}\right)  t_{j}m_{k}%
^{\prime}\left(  x_{j}\right)  \right]  -3\Sigma_{k}^{\prime}\left[  \left(
s_{j}-s_{k}\right)  b_{j}m_{k}^{\prime}\left(  x_{j}\right)  \right]  .\\
&  =-\frac{\beta_{j}b_{j}t_{j}}{6}-3\Sigma_{k}^{\prime}\left[  \left(
s_{j}-s_{k}\right)  b_{j}m_{k}^{\prime}\left(  x_{j}\right)  \right]  .
\end{align*}
Now applying this formula to the case of $\sigma=w_{1},\tau=w_{2},$ using
$\left(  \ref{zero}\right)  ,$ we conclude that
\[
\beta_{j}t_{j}-%
%TCIMACRO{\dsum \limits_{k\neq j}}%
%BeginExpansion
{\displaystyle\sum\limits_{k\neq j}}
%EndExpansion
\frac{36\left(  s_{j}-s_{k}\right)  }{\left(  \eta_{j}-\eta_{k}\right)  ^{3}%
}=0.
\]
Now applying the above computation to $\sigma=z_{1},\tau=z_{2},$ we find that
\[
\mathcal{K}%
\begin{bmatrix}
z_{1}\\
z_{2}%
\end{bmatrix}
=\left(  -3\Sigma_{j}\left(  \bar{s}_{j}m_{j}^{\prime}\right)  ,\Sigma
_{j}\left(  \bar{s}_{j}\beta_{j}m_{j}^{\prime}+\bar{t}_{j}n_{j}^{\prime
}\right)  \right)  ^{T},
\]
where $\left(  \bar{s}_{1},...,\bar{s}_{n},\bar{t}_{1},...,\bar{t}_{n}\right)
$ also lies in the tangent space of $M.$ With this at hand, we then can
compute the integral $Q$ using similar residue computation as for $T$, and
show that $Q$ equals zero, which implies that the vector
\[
\left(  I_{1},...,I_{n},B_{1},...,B_{n}\right)
\]
lies in the tangent space of $M$ at the point $\left(  \eta_{1},...,\eta
_{n},\beta_{1},...,\beta_{n}\right)  .$ The proof is thus completed.\bigskip\ 
\end{proof}

We remark that this result is consistent with the results in \cite{Shiota},
where the relation between CM hierarchy and the KP hierarchy is studied. Next
we show that the $k$ the flow is trivial, if $k$ is large.

\begin{lemma}
Let $n$ be fixed. Then for $k$ large, at $y=0,$ $X_{k}^{\left(  1\right)
}=0.$
\end{lemma}

\begin{proof}
By our choice of the parameter $a,$ if the index $k$ is an odd number, then
the main order term of $X_{k}$ is $O\left(  \frac{1}{x^{k+2}}\right)  .$ We
define
\[
\pi_{k}=\Sigma_{j}\eta_{j}^{k},\Pi_{k}=\Sigma_{j}\left(  \beta_{j}\eta_{j}%
^{k}\right)  .
\]
Since
\[
\frac{1}{\left(  1-t\right)  ^{2}}=\sum\limits_{k}\left(  k+1\right)  t^{k},
\]
we can write
\begin{align*}
q  &  =-\frac{3}{2}%
%TCIMACRO{\dsum \limits_{k=0}^{\infty}}%
%BeginExpansion
{\displaystyle\sum\limits_{k=0}^{\infty}}
%EndExpansion
\left(  x^{-k-2}\left(  \Sigma_{j}\eta_{j}^{k}\right)  \right)  ,\\
p  &  =\frac{1}{2}%
%TCIMACRO{\dsum \limits_{k=0}^{\infty}}%
%BeginExpansion
{\displaystyle\sum\limits_{k=0}^{\infty}}
%EndExpansion
\left(  x^{-k-2}\left(  \Sigma_{j}\beta_{j}\eta_{j}^{k}\right)  \right)  .
\end{align*}
It follows that
\[
X_{k}\left(  q,p\right)  =%
%TCIMACRO{\dsum \limits_{j=0}^{\infty}}%
%BeginExpansion
{\displaystyle\sum\limits_{j=0}^{\infty}}
%EndExpansion
\left(  -\frac{3\left(  j+1\right)  X_{k}^{\left(  1\right)  }\pi_{j}%
}{2x^{j+2}},\frac{\left(  j+1\right)  X_{k}^{\left(  2\right)  }\Pi_{j}%
}{2x^{j+2}}\right)  .
\]
Since the main order of $X_{k}$ is $x^{-k-2},$ we see that if $j<k,$ then
\[
X_{k}^{\left(  1\right)  }\pi_{j}=0\text{ and }X_{k}^{\left(  2\right)  }%
\Pi_{j}=0.
\]
However, if $k_{0}\geq n,$ then $\pi_{1},...,\pi_{k_{0}}$ form a basis of the
locus. From this we deduce that if
\[
n<k,
\]
then the first component of the flow $X_{k}$ is trivial.
\end{proof}

\subsection{Degree of the tau function}

\begin{lemma}
\label{l1}Suppose $\eta$ is a complex-valued homogeneous polynomial in $x,y$
of degree $m\ $and
\[
\left(  \mathfrak{D}_{x}^{2}+\mathfrak{D}_{y}^{2}\right)  \eta\cdot\eta=0.
\]
Then
\[
\eta\left(  x,y\right)  =a\left(  x^{2}+y^{2}\right)  ^{j}\left(  x+yi\right)
^{k},
\]
where $a$ is a constant and $2j+k=m.$ In particular, if $\eta$ is real-valued,
then $\eta=a\left(  x^{2}+y^{2}\right)  ^{m}$ for some real number $a.$
\end{lemma}

\begin{proof}
In the polar coordinate $\left(  r,\theta\right)  ,$ where $r=\sqrt
{x^{2}+y^{2}},$ we can write $\eta=r^{m}g\left(  \theta\right)  .$ Then
\begin{align*}
\left(  \mathfrak{D}_{x}^{2}+\mathfrak{D}_{y}^{2}\right)  \eta\cdot\eta &
=2\left(  \eta\Delta\eta-\left\vert \nabla\eta\right\vert ^{2}\right) \\
&  =2r^{m}g\left(  m^{2}r^{m-2}g+r^{m-2}g^{\prime\prime}\right)  -2\left(
m^{2}r^{2m-2}g^{2}+r^{2m-2}g^{\prime2}\right)  .
\end{align*}
From this we obtain
\[
gg^{\prime\prime}-g^{\prime2}=0,
\]
which implies $g\left(  \theta\right)  =ae^{b\theta}$ for some constants $a$
and $b.$ Since $g$ has to be $2\pi$-periodic in $\theta,$ we have $b=ki$ for
some integer $k.$ It follows that
\[
\eta=ar^{m}\left(  e^{i\theta}\right)  ^{k}=ar^{m-k}\left(  x+yi\right)
^{k}.
\]
Setting $j=\frac{m-k}{2},$ we arrive at the desired result.
\end{proof}

Let $\tau$ be a polynomial solution of the bilinear equation
\begin{equation}
\left(  \mathfrak{D}_{x}^{4}-\mathfrak{D}_{x}^{2}-\mathfrak{D}_{y}^{2}\right)
\tau\cdot\tau=0, \label{bilinear}%
\end{equation}
with $\deg\left(  \tau\right)  =m.$ By Lemma \ref{l1}, we can assume without
loss of generality that the highest degree terms of $\tau$ are of the form
\[
\left(  x^{2}+y^{2}\right)  ^{j}\left(  x+yi\right)  ^{k}=z^{j+k}\bar{z}%
^{j}:=\tau_{m},
\]
where $z=x+yi$ and $\bar{z}=x-yi$ are complex variables. Let us denote those
terms of $\tau$ with degree $m-1$ by $\tau_{m-1}.$ The previous lemma can also
be proved using the $\left(  z,\bar{z}\right)  $ coordinate. Indeed, we have
the following

\begin{lemma}
\label{eq:mth}$\tau_{m-1}=a_{1}z^{j+k-1}\bar{z}^{j}+a_{2}z^{j+k}\bar{z}^{j-1}$
for some constants $a_{1},a_{2}.$ In particular, if $\tau$ is real-valued,
then
\[
\tau_{m-1}=az^{j-1}\bar{z}^{j}+\bar{a}z^{j}\bar{z}^{j-1}.
\]

\end{lemma}

\begin{proof}
The terms of degree $2m-3$ in the left hand side of $\left(  \ref{bilinear}%
\right)  $ are of the form
\[
-\left(  \mathfrak{D}_{x}^{2}+\mathfrak{D}_{y}^{2}\right)  \tau_{m}\cdot
\tau_{m-1}.
\]
Suppose $z^{r}\bar{z}^{s}$ is a term appearing in $\tau_{m-1}$, then there
holds
\[
\mathfrak{D}_{z}\mathfrak{D}_{\bar{z}}\left(  z^{j+k}\bar{z}^{j}\right)
\cdot\left(  z^{r}\bar{z}^{s}\right)  =0.
\]
Direct computation tells us that
\[
\left[  \left(  j+k\right)  j-\left(  j+k\right)  s-jr+rs\right]
z^{j+k+r-1}\bar{z}^{j+s-1}=0.
\]
In the case of $k=0,$ we have
\[
j^{2}-js-jr+rs=0.
\]
That is, $r=j$ or $s=j.$

If the solution is real valued, then the degree $j+s$ term has to be
\begin{align*}
g\left(  x,y\right)   &  :=cz^{j}\bar{z}^{s}+\bar{c}z^{s}\bar{z}^{j}\\
&  =\left(  a+bi\right)  \left(  x+yi\right)  ^{j}\left(  x-yi\right)
^{s}+\left(  a-bi\right)  \left(  x+yi\right)  ^{s}\left(  x-yi\right)  ^{j}.
\end{align*}

Using $r+s=m-1$ and $2j+k=m,$ we obtain
\[
\left(  j+k\right)  j-\left(  j+k\right)  s-j\left(  m-1-s\right)  +\left(
m-1-s\right)  s=0.
\]
That is,
\[
-s^{2}+\left(  2j-1\right)  s-j\left(  j-1\right)  =0.
\]
Therefore, $s=j$ or $j-1.$

If in addition $\tau$ is real-valued, then $k=0$ and $\tau_{m}=z^{j}\bar
{z}^{j}.$ Hence
\[
\tau_{m-1}=az^{j-1}\bar{z}^{j}+\bar{a}z^{j}\bar{z}^{j-1}.
\]

\end{proof}

By this lemma, in the real-valued case, if we introduce new variables
$Z=z+\frac{a}{j}$ and $\bar{Z}=\bar{z}+\frac{\bar{a}}{j},$ then we see that
\[
z^{j}\bar{z}^{j}+az^{j-1}\bar{z}^{j}+\bar{a}z^{j}\bar{z}^{j-1}=Z^{j}\bar
{Z}^{j}+P,
\]
where $P$ is a polynomial of $Z,\bar{Z}$ with degree less than $j-1.$ This
means that we can find real numbers $b_{1},b_{2}$ such that in the new
variables $\tilde{x}=x+b_{1},\tilde{y}=y+b_{2},$ the highest degree term of
$\tau$ is $\left(  \tilde{x}^{2}+\tilde{y}^{2}\right)  ^{j}$ and $\tau$ does
not have terms with degree $2j-1.$

\begin{lemma}
Suppose $q=\frac{3}{2}\partial_{x}^{2}\ln\tau$ is a real valued rational
solution of the Boussinesq equation $\left(  \ref{eq:B}\right)  $, where
$\tau$ is a polynomial of degree $2n.$ Let $p=\int_{-\infty}^{x}\partial
_{y}qdx.$ Then for $x$ large, at $y=0,$
\[
q=-\frac{3n}{x^{2}}+O\left(  x^{-3}\right)  ,p=O\left(  \frac{1}{x^{5}
}\right)  .
\]

\end{lemma}

\begin{proof}
Since $\tau$ is real valued, after a possible translation of the
coordinate(and a scaling of the $y$ variable), it has the form
\[
\tau\left(  x,y\right)  =\left(  x^{2}+y^{2}\right)  ^{n}+
%TCIMACRO{\dsum \limits_{j+k\leq2n-2} }%
%BeginExpansion
{\displaystyle\sum\limits_{j+k\leq2n-2}}
%EndExpansion
\left(  a_{jk}x^{j}y^{k}\right)  .
\]
Therefore, $q=\frac{3n}{x^{2}}+O\left(  x^{-3}\right)  .$ We also have
\[
q=\frac{3n}{2}\partial_{x}^{2}\ln\left(  x^{2}+y^{2}\right)  +2\partial
_{x}^{2}\ln\left(  \frac{1}{\left(  x^{2}+y^{2}\right)  ^{n}}
%TCIMACRO{\dsum \limits_{j+k\leq2n-2} }%
%BeginExpansion
{\displaystyle\sum\limits_{j+k\leq2n-2}}
%EndExpansion
\left(  a_{jk}x^{j}y^{k}\right)  \right)  .
\]
Hence
\[
\partial_{y}q=2n\left(  \frac{i}{\left(  x+yi\right)  ^{3}}-\frac{i}{\left(
x-yi\right)  ^{3}}\right)  +O\left(  x^{-5}\right)  .
\]
The result then readily follows.
\end{proof}

\bigskip Now we are at a position to prove the following

\begin{theorem}
\bigskip Assume that $\tilde{q}$ is a real valued rational solution of the
Boussinesq equation $\left(  \ref{eq:B}\right)  $ with
\[
\tilde{q}\left(  x,y\right)  \rightarrow0,\text{ as }x^{2}+y^{2}
\rightarrow+\infty.
\]
Then $\tilde{q}=\frac{3}{2}\partial_{x}^{2}\ln\tau,$ where $\tau$ is a
polynomial in $x,y$ with degree $k\left(  k+1\right)  ,$ $k\in\mathbb{N}.$
\end{theorem}

\begin{proof}
Since $\left(  3a\right)  ^{2}+\frac{1}{48}=0,$ we compute
\[%
\begin{bmatrix}
-\frac{D}{4} & -3aD\\
-3aD & \frac{D}{12}%
\end{bmatrix}
\begin{bmatrix}
0 & D^{-1}\\
D^{-1} & 0
\end{bmatrix}
\begin{bmatrix}
-\frac{D}{4} & 3aD\\
3aD & \frac{D}{12}%
\end{bmatrix}
=0.
\]
This identity guarantees that if the main order term of the $X_{j}$ is
$x^{-k},$ then the main order term of $X_{j+2}$ will be at the order
$x^{-k-2}.$

For $x$ large, the main order term of $\tilde{q}$ is $\frac{m}{x^{2}}.$ Since
the degree of the polynomial is expected to be $k\left(  k+1\right)  ,$ we
expect $m$ to be $-\frac{3}{2}k\left(  k+1\right)  .$

We compute
\begin{align*}
&  \left(  D^{3}+qD+Dq\right)  \left(  \frac{1}{x^{j}}\right) \\
&  =\left[  -j\left(  j+1\right)  \left(  j+2\right)  -m\left(  j+j+2\right)
\right]  \frac{1}{x^{j+3}}\\
&  =-\left(  j+1\right)  \left(  j\left(  j+2\right)  +2m\right)  \frac
{1}{x^{j+3}}:=b_{m}\left(  j\right)  \frac{1}{x^{j+3}}.
\end{align*}
Similarly,
\begin{align*}
&  \left(  -\frac{5}{4}D^{3}-2\left(  qD+Dq\right)  \right)  \left(  \frac
{1}{x^{j}}\right) \\
&  =\frac{5}{4}j\left(  j+1\right)  \left(  j+2\right)  +2m\left(
j+j+2\right)  \frac{1}{x^{j+3}}\\
&  =\frac{1}{8}\left(  j+1\right)  \left(  10j\left(  j+2\right)  +32m\right)
\frac{1}{x^{j+3}}=:B_{m}\left(  j\right)  \frac{1}{x^{j+3}}.
\end{align*}
Vanishing of terms requires
\[
\frac{1}{4}b_{m}\left(  j\right)  -\frac{1}{4}B_{m}\left(  j\right)  =0.
\]
That is,
\[
m=-\frac{3}{8}j\left(  j+2\right)  .
\]
Let $j=2k,$ we find that
\[
m=-\frac{3}{2}k\left(  k+1\right)  .
\]

This completes the proof.
\end{proof}

Summarizing the previous discussion, we conclude that Theorem \ref{main} is proved.

Now suppose $q$ is a solution of the equation%
\begin{equation}
\partial_{y}^{2}q=3\partial_{x}^{2}\left(  \partial_{x}^{2}q+4q^{2}-q\right)
. \label{Bnew}%
\end{equation}
The energy of $q$ is
\[
H\left(  q\right)  :=\int_{\mathbb{R}^{2}}\left[  \frac{3}{2}\left\vert
\partial_{x}q\right\vert ^{2}+4q^{3}-\frac{3}{2}q^{2}-\left\vert \partial
_{x}^{-1}\partial_{y}q\right\vert ^{2}\right]  .
\]
We now know that $q$ has the form $\frac{3}{2}\partial_{x}^{2}\ln\tau,$ where
$\tau$ is a polynomial with degree $k\left(  k+1\right)  .$ The classical lump
solution for $\left(  \ref{Bnew}\right)  $ is
\[
u_{0}\left(  x,y\right)  =\frac{3}{2}\partial_{x}^{2}\ln\left(  x^{2}%
+3y^{2}+3\right)  .
\]
Note that up to a translation in the $x$ and $y$ variables, the tau function
with degree $2$ is unique.

Following the same proof as that of the appendix of
Gorshkov-Pelinovskii-Stepanyants \cite{Pelinovskii93}(see equation (A6)
there), we obtain%
\begin{equation}
H\left(  q\right)  =\frac{k\left(  k+1\right)  }{2}H\left(  u_{0}\right)  .
\label{energyq}%
\end{equation}
We also know from \cite{Saut97} that the equation $\left(  \ref{Bnew}\right)
$ has variational structure and possess a ground state. From the energy
quantization identity $\left(  \ref{energyq}\right)  ,$ we infer immediately
that the lump solution is the unique ground state.

\section{The analysis of even solutions}

In this section, we would like to analyze the even solutions of the Boussinesq
equation. Combining our classification result obtained in the previous section
with the existence result of \cite{Pelinovskii} mentioned in Section 2, we
find that these solutions exist if and only if their tau functions are
polynomials of degree $2n=k\left(  k+1\right)  .$ From the semilinear elliptic
PDE point of view, these solutions should play similar role as the radially
symmetric solutions of the Schrodinger equation.

Now suppose $q$ is an even solution$.$ From Lemma \ref{l1}, we can assume that
the sum of the degree $2n$ terms of $\tau$ is $T_{n,0}=\left(  x^{2}%
+y^{2}\right)  ^{n}.$ We also denote the sum of the degree $2n-2j$ terms of
$\tau$ by $T_{n,j}.$

Let us define functions
\[
g_{j}:=\left(  x^{2}+y^{2}\right)  ^{n-3j}x^{2j}y^{2j}\text{, \ }\xi
_{i,j}:=\left(  \mathfrak{D}_{x}^{2}+\mathfrak{D}_{y}^{2}\right)  g_{i}\cdot
g_{j}.
\]
Observe that actually $\xi_{i,j}$ can be divided by $\left(  x^{2}%
+y^{2}\right)  ^{2n-3i-3j-1}.$ We introduce the constants%

\[
d_{i,j}:=\frac{\left(  \mathfrak{D}_{x}^{2}+\mathfrak{D}_{y}^{2}\right)
g_{i}\cdot g_{j}}{\left(  x^{2}+y^{2}\right)  ^{2n-3i-3j-1}}\bigg|_{\left(
x^{2}=-1,y^{2}=1\right)  }.
\]

\begin{lemma}
\bigskip$d_{i,j}=-12\left(  i-j\right)  ^{2}\left(  -1\right)  ^{i+j}.$
\end{lemma}

\begin{proof}
This follows from straightforward computation. We do it below for
completeness. We have
\[
\left(  \mathfrak{D}_{x}^{2}+\mathfrak{D}_{y}^{2}\right)  g_{i}\cdot
g_{j}=\Delta g_{i}g_{j}+g_{i}\Delta g_{j}-2\nabla g_{i}\cdot\nabla g_{j}.
\]
Note that
\begin{align*}
\Delta g_{j} &  =\left(  2n-6j\right)  ^{2}r^{2n-6j-2}x^{2j}y^{2j}+8j\left(
2n-6j\right)  r^{2n-6j-2}x^{2j}y^{2j}+r^{2n-6j}\Delta\left(  x^{2j}%
y^{2j}\right)  \\
&  =\left(  2n+2j\right)  \left(  2n-6j\right)  r^{2n-6j-2}x^{2j}%
y^{2j}+r^{2n-6j}\Delta\left(  x^{2j}y^{2j}\right)  .
\end{align*}%
\begin{align*}
&  \nabla g_{i}\cdot\nabla g_{j}\\
&  =\left(  x^{2}+y^{2}\right)  ^{2n-3i-3j-2}x^{2i+2j-2}y^{2i+2j}\left(
\left(  2n-6i\right)  x^{2}+2i\left(  x^{2}+y^{2}\right)  \right)  \left(
\left(  2n-6j\right)  x^{2}+2j\left(  x^{2}+y^{2}\right)  \right)  \\
&  +\left(  x^{2}+y^{2}\right)  ^{2n-3i-3j-2}x^{2i+2j}y^{2i+2j-2}\left(
\left(  2n-6i\right)  y^{2}+2i\left(  x^{2}+y^{2}\right)  \right)  \left(
\left(  2n-6j\right)  y^{2}+2j\left(  x^{2}+y^{2}\right)  \right)  .
\end{align*}
Therefore,
\begin{align*}
&  \frac{\left(  \mathfrak{D}_{x}^{2}+\mathfrak{D}_{y}^{2}\right)  g_{i}\cdot
g_{j}}{\left(  x^{2}+y^{2}\right)  ^{2n-3i-3j-1}}|_{x^{2}=-1,y^{2}=1}\\
&  =\left(  2n+2i\right)  \left(  2n-6i\right)  \left(  -1\right)  ^{i+j}\\
&  +\left(  2n+2j\right)  \left(  2n-6j\right)  \left(  -1\right)  ^{i+j}\\
&  -2\left(  \left(  2n-6i\right)  \left(  2n-6j\right)  +2\left(
2n-6i\right)  2j+2\left(  2n-6j\right)  2i\right)  \left(  -1\right)  ^{i+j}\\
&  =-12\left(  i-j\right)  ^{2}\left(  -1\right)  ^{i+j}.
\end{align*}
This finishes the proof.
\end{proof}

Let us now consider the function $\mathfrak{D}_{x}^{4}\left(  x^{2}%
+y^{2}\right)  ^{n-3i}\cdot\left(  x^{2}+y^{2}\right)  ^{n-3j}.$ Since we have
taken the fourth order derivative, this function is dividable by $\left(
x^{2}+y^{2}\right)  ^{2n-3i-3j-4}.$

We also define
\[
p_{i,j}=\frac{\mathfrak{D}_{x}^{4}g_{i}\cdot g_{j}}{\left(  x^{2}%
+y^{2}\right)  ^{2n-3i-3j-4}}\bigg|_{\left(  x^{2}=-1,y^{2}=1\right)  .}%
\]
These constants depending on $n,i,j.$ Explicitly, $\left(  -1\right)
^{i+j}p_{i,j}$ is equal to
\begin{align*}
&  1296i^{4}-5184i^{3}j+7776i^{2}j^{2}-5184ij^{3}+1296j^{4}+2592i^{3}%
-2592i^{2}j\\
&  -1728i^{2}n-2592ij^{2}+3456ijn+2592j^{3}-1728j^{2}n+1584i^{2}-1440ij\\
&  -576in+1584j^{2}-576jn+192n^{2}+288i+288j-192n.
\end{align*}
In the special case $i=j,$
\[
p_{i,j}=\left(  -1\right)  ^{i+j}192\left(  3j-n+1\right)  \left(
3j-n\right)  .
\]
When $i=j-1,$ we have $p_{i,j}=\left(  -1\right)  ^{i+j}192\left(
3j-n+7\right)  \left(  3j-n\right)  .$ If $i=j-2,$ then
\[
p_{i,j}=\left(  -1\right)  ^{i+j}192\left(  3j-n+30\right)  \left(
3j-n+1\right)  .
\]
Moreover, if $i=0,$ then
\[
p_{i,j}=1296j^{4}+2592j^{3}-1728j^{2}n+1584j^{2}-576jn+192n^{2}+288j-192n.
\]

Now we would like to define a sequence of numbers $a_{m},m=0,1,...,$ depending
on $n,$ in the following way.

Take $a_{0}=1.$ Then $a_{m}$ is determined by $a_{1},...,a_{m-1}$ through the
following recursive relation:
\[
\sum\limits_{i,j\leq m,i+j=m}\left(  a_{i}a_{j}d_{i,j}\right)  =\sum
\limits_{i,j\leq m,i+j=m-1}\left(  a_{i}a_{j}p_{i,j}\right)  .
\]

We regard $a_{j}$ as a polynomial of the variable $n.$ Now let us define the
constant
\[
J_{n}:=\sum\limits_{i,j\leq\lbrack n/3],i+j=[n/3]+1}\left(  a_{i}a_{j}%
d_{i,j}\right)  -\sum\limits_{i,j\leq\lbrack n/3],i+j=[n/3]}\left(  a_{i}%
a_{j}p_{i,j}\right)  .
\]
Here $\left[  m\right]  $ denotes the largest integer which does not exceed
$m.$

\begin{proposition}
Let $n$ be a fixed integer. If $J_{n}\neq0,$ then the Boussinesq equation has
no rational even solution with degree $2n.$
\end{proposition}

\begin{proof}
First of all, we claim that $T_{n,j}$ has the form
\[
a_{j}\left(  x^{2}+y^{2}\right)  ^{n-3j}x^{2j}y^{2j}+\left(  x^{2}
+y^{2}\right)  ^{n-3j+1}\Gamma\left(  x,y\right)  ,
\]
where $\Gamma$ is a homogeneous polynomial in $x,y$ with degree $4j-2.$ Indeed,

Let us denote the function $\left(  \mathfrak{D}_{x}^{4}-\mathfrak{D}_{x}%
^{2}-\mathfrak{D}_{y}^{2}\right)  f\cdot f$ by $K_{f}.$ Since we have chosen
$a_{0}$ to be $1,$ $K_{f}$ is a polynomial of degree at most $4n-4.$ The terms
with degree $4n-4$ are given by
\[
\mathfrak{D}_{x}^{4}T_{n,0}\cdot T_{n,0}-\left(  \mathfrak{D}_{x}%
^{2}+\mathfrak{D}_{y}^{2}\right)  T_{n,0}\cdot T_{n,1}.
\]
This function is dividable by $\left(  x^{2}+y^{2}\right)  ^{2n-4}.$ We write
is as
\[
b_{1}\left(  x^{2}+y^{2}\right)  ^{2n-4}x^{2}y^{2}+\left(  x^{2}+y^{2}\right)
^{2n-3}M\left(  x,y\right)  .
\]
Inserting $x^{2}=-1,y^{2}=1$ into this function, we find that necessary
$b_{1}=0.$ Therefore, we get
\[
a_{0}^{2}p_{0,0}-a_{0}a_{1}d_{0,1}=0.
\]
Similarly, consider the terms with degree $4n-6,$ we get
\[
\mathfrak{D}_{x}^{4}T_{n,0}\cdot T_{n,1}-\left(  \mathfrak{D}_{x}%
^{2}+\mathfrak{D}_{y}^{2}\right)  T_{n,1}\cdot T_{n,1}-\left(  \mathfrak{D}%
_{x}^{2}+\mathfrak{D}_{y}^{2}\right)  T_{n,0}\cdot T_{n,2}=0.
\]
Then
\[
a_{0}a_{1}p_{0,1}-a_{1}^{2}d_{1,1}-a_{0}a_{2}d_{0,2}=0.
\]
Similarly, for $m\leq\left[  n/3\right]  ,$
\[
\sum\limits_{i,j\leq m,i+j=m}\left(  a_{i}a_{j}d_{i,j}\right)  =\sum
\limits_{i,j\leq m,i+j=m-1}\left(  a_{i}a_{j}p_{i,j}\right)  .
\]
Since we require that the solution is a polynomial, the function
\[
\sum\limits_{i,j\leq\lbrack n/3],i+j=[n/3]+1}a_{i}a_{j}\left(  \mathfrak{D}%
_{x}^{2}+\mathfrak{D}_{y}^{2}\right)  T_{n,i}\cdot T_{n,j}-\sum
\limits_{i,j\leq\left[  n/3\right]  ,i+j=\left[  n/3\right]  }a_{i}%
a_{j}\mathfrak{D}_{x}^{4}T_{n,i}\cdot T_{n,j}%
\]
should be dividable by $\left(  x^{2}+y^{2}\right)  ^{n-1},$ this implies that
\ $J_{n}=0.$
\end{proof}

We have computed the constants $a_{j}$ and $J_{n},$ using software like
\textquotedblleft\textit{Mathematica}\textquotedblright. It turns out that at
least for $n\leq300,$ $J_{n}$ is equal to zero if and only if $n=\frac
{k\left(  k+1\right)  }{2}$ for some integer $k.$

The previous analysis can also be interpreted in terms of $z$ and $\bar{z}$
variables. Let us explain this in more details.

From the proof of Lemma \ref{eq:mth}, we know that if $\eta$ satisfies
\[
\mathfrak{D}_{z}\mathfrak{D}_{\bar{z}}T_{n,0}\cdot\eta=0,
\]
then for some constants $c_{1},c_{2},$%
\[
\eta=c_{1}z^{n}\bar{z}^{s}+c_{2}z^{m}\bar{z}^{n}.
\]
This also tells us that the equation
\[
\mathfrak{D}_{z}\mathfrak{D}_{\bar{z}}T_{n,0}\cdot\eta=z^{\alpha}\bar
{z}^{\beta},
\]
is not solvable if either $\alpha$ or $\beta$ equal $2n-1.$ Since
$T_{n,0}=z^{n}\bar{z}^{n},$ another necessary condition is that
\[
\alpha\geq n-1\text{ and }\beta\geq n-1.
\]

\begin{lemma}
The $T_{n,1}$ term has the following form:%
\begin{align*}
T_{n,1}  &  =\frac{1}{2}\left(  n-n^{2}\right)  z^{n+1}\bar{z}^{n-3}%
+3n^{2}z^{n-1}\bar{z}^{n-1}+\frac{1}{2}\left(  n-n^{2}\right)  z^{n-3}\bar
{z}^{n+1}\\
&  +cz^{n}\bar{z}^{n-2}+cz^{n-2}\bar{z}^{n},
\end{align*}
where $c$ is a real valued constant.
\end{lemma}

\begin{proof}
We compute%
\[
\mathfrak{D}_{x}^{4}T_{n,0}\cdot T_{n,0}=\left(  12n^{2}-12n\right)
z^{2n}\bar{z}^{2n-4}+24n^{2}z^{2n-2}\bar{z}^{2n-2}+\left(  12n^{2}-12n\right)
z^{2n}\bar{z}^{2n-4}.
\]
Since our solution is even, the conclusion then follows from the fact that
$T_{n,1}$ solves the equation(Note that the constant is $8$, rather than $4$)
\[
8\mathfrak{D}_{z}\mathfrak{D}_{\bar{z}}T_{n,0}\cdot T_{n,1}=\mathfrak{D}%
_{x}^{4}T_{n,0}\cdot T_{n,0}.
\]
The fact the our solution is real and even forces the coefficients before
$z^{n}\bar{z}^{n-2}$ and $z^{n-2}\bar{z}^{n}$ to be a same real constant. This
completes the proof.
\end{proof}

We emphasize that in general the constant $c$ will not be zero. For instance,
the degree $12$ solution obtained in \cite{Pelinovskii} is
\begin{align*}
&  (x^{2}+y^{2})^{6}+2(x^{2}+y^{2})^{3}(49x^{4}+198x^{2}y^{2}+29y^{4})\\
&  +5(147x^{8}+3724x^{6}y^{2}+7490x^{4}y^{4}+7084x^{2}y^{6}+867y^{8})\\
&  +\frac{140}{3}(539x^{6}+4725x^{4}y^{2}-315x^{2}y^{4}+5707y^{6})\\
&  +\frac{1225}{9}(391314x^{2}-12705x^{4}+4158x^{2}y^{2}+40143y^{4}%
+736890x^{2}+717409).
\end{align*}
It can also be written as
\begin{align*}
&  z^{6}\bar{z}^{6}-15z^{7}\bar{z}^{3}+10z^{6}\bar{z}^{4}+108z^{5}\bar{z}%
^{5}+10z^{4}\bar{z}^{6}-15z^{3}\bar{z}^{7}\\
&  -45z^{8}+150z^{7}\bar{z}-875z^{6}\bar{z}^{2}-1050z^{5}\bar{z}^{3}%
+4375z^{4}\bar{z}^{4}-1050z^{3}\bar{z}^{5}-875z^{2}\bar{z}^{6}+150z\bar{z}%
^{7}-45\bar{z}^{8}\\
&  -22330z^{6}/3+20895z^{5}\bar{z}-52850z^{4}\bar{z}^{2}+103950z^{3}\bar
{z}^{3}-52850z^{2}\bar{z}^{4}+20895z\bar{z}^{5}-22330\bar{z}^{6}/3\\
&  +594125z^{4}/3-1798300z^{3}\bar{z}+1471225z^{2}\bar{z}^{2}-1798300z\bar
{z}^{3}+594125z^{4}/3\\
&  +38390275z^{2}+76780550z\bar{z}+38390275\bar{z}^{2}+878826025/9.
\end{align*}

As a polynomial of variables $z,\bar{z},$ the total degree of the homogeneous
polynomial $T_{n,j}$ is equal to $2n-2j.$ For each fixed $j,$ inspecting the
term in $T_{n,j}$ with lowest degree in $\bar{z},$ we find that it has to be
of the form $\sigma_{j}z^{n+j}\bar{z}^{n-3j}.$ Indeed, the constants
$\sigma_{j}$ can be defined recursively and uniquely by the following
equation: For $j=1,...,$
\begin{align*}
&  8\mathfrak{D}_{z}\mathfrak{D}_{\bar{z}}T_{n,0}\cdot\left(  \sigma
_{j}z^{n+j}\bar{z}^{n-3j}\right) \\
&  =%
%TCIMACRO{\dsum \limits_{k+m=j-1}}%
%BeginExpansion
{\displaystyle\sum\limits_{k+m=j-1}}
%EndExpansion
\left[  \mathfrak{D}_{x}^{4}\left(  \left(  \sigma_{k}z^{n+k}\bar{z}%
^{n-3k}\right)  \right)  \cdot\left(  \sigma_{m}z^{n+m}\bar{z}^{n-3m}\right)
\right] \\
&  -4%
%TCIMACRO{\dsum _{k+m=j}}%
%BeginExpansion
{\displaystyle\sum_{k+m=j}}
%EndExpansion
\left[  \mathfrak{D}_{z}\mathfrak{D}_{\bar{z}}\left(  \left(  \sigma
_{k}z^{n+k}\bar{z}^{n-3k}\right)  \right)  \cdot\left(  \sigma_{m}z^{n+m}%
\bar{z}^{n-3m}\right)  \right]  .
\end{align*}
Observe that the degree of $\mathfrak{D}_{z}\mathfrak{D}_{\bar{z}}T_{n,0}%
\cdot\left(  z^{n+j}\bar{z}^{n-3j}\right)  $ is equal to $z^{2n+j-1}\bar
{z}^{2n-3j-1}.$ However, as discussed above, the equation
\begin{equation}
\mathfrak{D}_{z}\mathfrak{D}_{\bar{z}}T_{n,0}\cdot\eta=z^{2n+j-1}\bar
{z}^{2n-3j-1} \label{kernel}%
\end{equation}
will not be solvable if
\[
2n-3j-1<n-1.
\]
That is, $n<3j.$ This means that it necessary condition for an even solution
to exist is
\[
\sigma_{j_{0}}=0,\text{ for }j_{0}=\left[  \frac{n}{3}\right]  +1.
\]

We have also verified that for $0\leq n\leq300,$ $\sigma_{n}$ equals zero if
and only if $n=k\left(  k+1\right)  /2$ for some integer $k.$

This algorithm inspires us to study the uniqueness of even solution. The
possible nonuniqueness arises from the fact that equation $\left(
\ref{kernel}\right)  $ has kernels of the form $z^{n}\bar{z}^{n-2q}%
+z^{n-2q}\bar{z}^{n}.$ Note that for each fixed $q=1,...,\left[  n/2\right]
,$ the lowest possible degree term generated by the function $z^{n}\bar
{z}^{n-2q}$ in $T_{n,q+j}$ is of the form $\beta_{j}z^{n+j}\bar{z}^{n-2q-3j}.$
Here $\beta_{0}=1,$ and similar to $\sigma_{j},$ for $j\geq1,$ the sequence
$\beta_{j}$ is determined by the following recursive formula:
\begin{align*}
&  4\mathfrak{D}_{z}\mathfrak{D}_{\bar{z}}\left(  z^{n}\bar{z}^{n}\right)
\cdot\left(  \beta_{j}z^{n+j}\bar{z}^{n-2q-3j}\right) \\
&  =%
%TCIMACRO{\dsum \limits_{k+m=j-1}}%
%BeginExpansion
{\displaystyle\sum\limits_{k+m=j-1}}
%EndExpansion
\left[  \mathfrak{D}_{x}^{4}\left(  \left(  \sigma_{k}z^{n+k}\bar{z}%
^{n-3k}\right)  \right)  \cdot\left(  \beta_{m}z^{n+m}\bar{z}^{n-2q-3m}%
\right)  \right] \\
&  -4%
%TCIMACRO{\dsum _{k+m=j}}%
%BeginExpansion
{\displaystyle\sum_{k+m=j}}
%EndExpansion
\left[  \mathfrak{D}_{z}\mathfrak{D}_{\bar{z}}\left(  \left(  \sigma
_{k}z^{n+k}\bar{z}^{n-3k}\right)  \right)  \cdot\left(  \beta_{m}z^{n+m}%
\bar{z}^{n-2q-3m}\right)  \right]  .
\end{align*}
Note that $\beta_{j}$ are also depending on $q.$ The degree of $\bar{z}$ in
$\mathfrak{D}_{z}\mathfrak{D}_{\bar{z}}\left(  z^{n}\bar{z}^{n}\right)
\cdot\left(  \beta_{j}z^{n+j}\bar{z}^{n-2q-3j}\right)  $ is $2n-2q-3j-1.$ For
\[
j=\bar{j}:=\left[  \left(  n-2q\right)  /3\right]  +1,
\]
there holds%
\[
2n-2q-3j-1<n-1.
\]
We then define, for $q=1,...,\left[  n/2\right]  ,$
\[
\gamma_{q}:=\beta\left(  \bar{j}\right)  .
\]
We have the following:

\begin{lemma}
For given $n=k\left(  k+1\right)  /2,$ if $\gamma_{q}\neq0$ for all
$q=1,...,\left[  n/2\right]  ,$ then the even solution is unique.
\end{lemma}

\begin{proof}
Note that the kernel terms $z^{n}\bar{z}^{n-2q}+z^{n-2q}\bar{z}^{n}$ are the
only possible sources of nonuniqueness. We consider them for each $q,$
starting from $q=1.$

Since $\gamma_{1}\neq0,$ we see that the coefficient of $z^{n}\bar{z}%
^{n-2}+z^{n-2}\bar{z}^{n}$ is uniquely determined in $T_{n,1},$ otherwise one
of the equations for the terms in $T_{n,\left[  \left(  n-2\right)  /3\right]
+2}$ will not be solvable. Once $z^{n}\bar{z}^{n-2}+z^{n-2}\bar{z}^{n}$ is
determined, we use the assumption that $\gamma_{2}\neq0$ to conclude that we
don't have the freedom to choose the kernel $z^{n}\bar{z}^{n-4}+z^{n-4}\bar
{z}^{n}$ in $T_{n,2}$. Proceeding with this argument, we see that all the
kernel terms $z^{n}\bar{z}^{n-2q}+z^{n-2q}\bar{z}^{n}$ are uniquely
determined. This finishes the proof.
\end{proof}

We can compute the precise value of the constant $\gamma_{q}$ explicitly for
each $n$(using \textquotedblleft\textit{Mathematica}\textquotedblright). It
turns out that for $n=k\left(  k+1\right)  /2\leq300,$ all the constants
$\gamma_{q}$ are nonzero. For example, when $n=15,$ we have%
\begin{align*}
\gamma_{1}  &  =\frac{3219950475}{374},\gamma_{2}=-\frac{800391375}%
{416},\gamma_{3}=24045525/4,\\
\gamma_{4}  &  =\frac{34505100}{187},\gamma_{5}=-\frac{74025}{52},\gamma
_{6}=\frac{55335}{2},\gamma_{7}=-\frac{5460}{17}.
\end{align*}
It should be pointed out that all these computations are actually rigorous. We
are therefore arriving at the following: ;

\begin{theorem}
Suppose $\tau$ is a polynomial of degree $2n$ with real coefficients
satisfying
\[
\tau\left(  x,y\right)  =\tau\left(  x,-y\right)  =\tau\left(  -x,y\right)
\]
and
\[
\left(  \mathfrak{D}_{x}^{2}+\mathfrak{D}_{y}^{2}-\mathfrak{D}_{x}^{4}\right)
\tau\cdot\tau=0.
\]
Assume $n=k\left(  k+1\right)  /2\leq300$ for some positive integer $k.$ Then
$\tau$ is unique, up to a multiplicative constant.
\end{theorem}

The upbound $300$ can be significantly improved. We actually expect that the
uniqueness of even solution holds for all $n\in\mathbb{N}$(By our result, $n$
has to be $k\left(  k+1\right)  /2$). The fully proof of this uniqueness
result seems to be a challenging problem at this moment.

\end{document}